\documentclass[11pt]{amsart}
\usepackage[margin=1.5in]{geometry} 
\usepackage{amsmath,amssymb}
\usepackage{enumitem}
\usepackage{xcolor}
\usepackage{graphicx}
\usepackage{dsfont}
\usepackage{tikz-cd} 
\usepackage[normalem]{ulem}

\newtheorem{theorem}{Theorem}[section]
\newtheorem{proposition}[theorem]{Proposition}
\newtheorem{lemma}[theorem]{Lemma}
\newtheorem{corollary}[theorem]{Corollary}
\theoremstyle{definition}

\newtheorem{remark}[theorem]{Remark}

\newtheorem{step}{Step}
\theoremstyle{remark}

\newcommand{\R}{\mathbb{R}}

\newcommand{\N}{\mathbb{N}}

\renewcommand{\epsilon}{\varepsilon}

\newcommand{\MNg}[1]{}%

\newcommand{\cC}{\mathcal{C}}
\newcommand{\cJ}{\mathcal{J}}
\newcommand{\cN}{\mathcal{N}}
\newcommand{\cP}{\mathcal{P}}
\newcommand{\cL}{\mathcal{L}}
\newcommand{\X}{\mathcal{X}}
\newcommand{\Y}{\mathcal{Y}}
\newcommand{\as}{\mbox{-a.s.}}
\newcommand{\Xcpt}{\mathcal{X}_{\mathrm{cpt}}}
\newcommand{\Ycpt}{\mathcal{Y}_{\mathrm{cpt}}}
\newcommand{\eps}{\varepsilon}
\newcommand{\qandq}{\quad\mbox{and}\quad}
\newcommand{\qforallq}{\quad\mbox{for all}\quad}

\newcommand{\1}{\mathbf{1}}
\DeclareMathOperator*{\argmin}{arg\, min}

\numberwithin{equation}{section}

\usepackage[pdfborder={0 0 0}]{hyperref}
\hypersetup{
  urlcolor = black,
  pdfauthor = {Marcel Nutz, Johannes Wiesel},
  pdfkeywords = {Optimal Transport; Entropic Regularization; Schrodinger potentials; Sinkhorn's Algorithm},
  pdftitle = {Stability of Schrodinger Potentials and Convergence of Sinkhorn's Algorithm},
  pdfsubject = {Stability of Schrodinger Potentials and Convergence of Sinkhorn's Algorithm},
  pdfpagemode = UseNone
}

\begin{document}

\author{Marcel Nutz}
\thanks{MN acknowledges support by an Alfred P.\ Sloan Fellowship and NSF Grants DMS-1812661,  DMS-2106056.}
\address[MN]{Departments of Statistics and Mathematics, Columbia University,  1255 Amsterdam Avenue, New York, NY 10027, USA}
\email{mnutz@columbia.edu}

\author{Johannes Wiesel}
\address[JW]{Department of Statistics, Columbia University, 1255 Amsterdam Avenue,
New York, NY 10027, USA}
\email{johannes.wiesel@columbia.edu}

\title[Stability of Schr\"odinger Potentials]{Stability of Schr\"odinger Potentials\\and Convergence of Sinkhorn's Algorithm}

\date{\today}

\begin{abstract}
  We study the stability of  entropically regularized optimal transport  with respect to the marginals. Given marginals converging weakly, we establish a strong convergence for the Schr\"odinger potentials describing the density of the optimal couplings. When the  marginals converge in total variation, the optimal couplings also converge in total variation. This is applied to show that Sinkhorn's algorithm converges in total variation when costs are quadratic and marginals are subgaussian, or more generally for all continuous costs satisfying an integrability condition.
\end{abstract}


\maketitle

\section{Introduction}\label{se:intro}

Let $(\X,\mu)$ and $(\Y,\nu)$ be Polish probability spaces and $\Pi(\mu,\nu)$ the set of all couplings; that is, probability measures~$\pi$ on~$\X\times\Y$ with marginals~$(\mu,\nu)$. Moreover, let $c:\X\times\Y\to\R_{+}$ be continuous. The entropic optimal transport problem with regularization parameter~$\eps\in(0,\infty)$ is
\begin{equation}\label{eq:epsOT}
  \cC_{\eps}(\mu,\nu):=\inf_{\pi\in\Pi(\mu,\nu)} \int_{\X\times\Y} c(x,y) \, \pi(dx,dy) + \eps H(\pi|\mu\otimes\nu), 
\end{equation}
where $H(\,\cdot\,|\mu\otimes\nu)$ denotes relative entropy with respect to the product of the marginals,
$$
  H(\pi|\mu\otimes\nu):=\begin{cases}
\int \log \frac{d\pi}{d(\mu\otimes\nu)} \,d\pi, & \pi\ll \mu\otimes\nu,\\
\infty, & \pi\not\ll \mu\otimes\nu.
\end{cases} 
$$

Entropic optimal transport traces back to the Schr\"odinger bridge problem associated with Schr\"odinger's thought experiment~\cite{Schrodinger.31} on the most likely evolution of a particle system, see~\cite{Follmer.88, Leonard.14} for surveys. More recently, popularized by~\cite{Cuturi.13}, the problem has received immense interest as an approximation of the (unregularized) Monge--Kantorovich optimal transport problem corresponding to $\eps=0$, especially for computing the 2-Wasserstein distance in high-dimensional applications such as machine learning, statistics, image and language processing (e.g., \cite{AlvarezJaakkola.18, WGAN.17,ChernozhukovEtAl.17,RubnerTomasiGuibas.00}). As a result, the cost of principal interest  is~$c(x,y)=\|x-y\|^{2}$ on~$\R^{d}\times\R^{d}$ and the convergence properties as $\eps\to0$ have been studied in detail; see \cite{
Berman.20, BlanchetJambulapatiKentSidford.18, ChenGeorgiouPavon.16, CominettiSanMartin.94, ConfortiTamanini.19, GigliTamanini.21, Leonard.12, Mikami.02, Mikami.04, NutzWiesel.21, Pal.19, Weed.18}, among others.
The main appeal of~\eqref{eq:epsOT} in this computational context is that it can be solved efficiently and at large scale by Sinkhorn's algorithm; see \cite{PeyreCuturi.19} and its numerous references. The algorithm is initialized  at the probability measure $\pi_{-1}\propto e^{-c/\eps}d(\mu\otimes\nu)$ and  its iterates are defined for $t\geq0$ by
\begin{align}\label{eq:SinkPrimalImplicit}
    \pi_{2t} := \argmin_{\Pi(\ast,\nu)} H(\,\cdot\,|\pi_{2t-1}), \qquad
    \pi_{2t+1} := \argmin_{\Pi(\mu,\ast)} H(\,\cdot\,|\pi_{2t}), 
\end{align}
where $\Pi(\ast,\nu)$ is the set of measures on $\X\times\Y$ with second marginal~$\nu$ (and arbitrary first marginal), and $\Pi(\mu,\ast)$ is defined analogously. The algorithm  alternatingly ``fits'' the marginals~$\mu$ and~$\nu$, hence is also called iterative proportional fitting procedure (IPFP). The argmin in~\eqref{eq:SinkPrimalImplicit} can be solved explicitly, and then each step of the algorithm only requires an explicit integration against~$e^{-c/\eps}$, cf.\ Section~\ref{se:sink}.
The algorithm dates back as far as~\cite{DemingStephan.40} and its convergence properties are well studied when the cost~$c$ is bounded: the convergence of $\pi_{n}$ to the solution~$\pi_{*}$ of~\eqref{eq:epsOT} holds in total variation (and in relative entropy),  see \cite{Carlier.21, ChenGeorgiouPavon.16, FranklinLorenz.89, IrelandKullback.68,Kullback.68,Ruschendorf.95,Sinkhorn.64,SinkhornKnopp.67}, among others. More precisely, \cite{Ruschendorf.95} relaxed the boundedness condition, but introduced several other conditions, including one (see~(B1) in~\cite{Ruschendorf.95}) that essentially forces $c$  to be bounded from above in one variable and thus excludes quadratic cost with unbounded marginal supports.

One main result of this paper is the convergence $\|\pi_{n}-\pi_{*}\|_{TV}\to0$ of Sinkhorn's algorithm~\eqref{eq:SinkPrimalImplicit} for quadratic cost and arbitrary subgaussian marginals. More generally, our result (see Corollary~\ref{co:sink}) holds for the continuous cost~$c$ and marginals~$(\mu,\nu)$ as soon  as
  \begin{equation}\label{eq:sinkExpCondIntro}
    \mbox{$e^{\beta c}\in L^{1}(\mu\otimes\nu)$ \;for some $\beta>0$}.
  \end{equation}
A fairly elementary proof of the  convergence $\|\pi_{n}-\pi_{*}\|_{TV}\to0$, following the same idea of weak-star compactness as~\cite{Ruschendorf.95}, was recently given in~\cite{Nutz.20} under the condition that~\eqref{eq:sinkExpCondIntro} holds for some $\beta>\eps^{-1}$. We emphasize that this condition  does not capture the regime of principal interest: when approximating the 2-Wasserstein distance and marginals are standard Gaussians, say, this condition forces $\eps>1/2$, but the approximation often requires~$\eps$ to be several orders of magnitude smaller (e.g., \cite{PeyreCuturi.19}). While the case $\beta>\eps^{-1}$ is broadly similar to the case of bounded cost, it seems that the regime $\beta<\eps^{-1}$  requires a fundamentally different line of attack---which brings  us to the main theorem of this paper.

In all that follows, we focus on $\eps=1$ in~\eqref{eq:epsOT} for notational simplicity; the general case is easily retrieved by replacing~$c$ with~$c/\eps$. Assuming that
\begin{equation}\label{eq:EOT}
  \cC(\mu,\nu):=\inf_{\pi\in\Pi(\mu,\nu)} \int_{\X\times\Y} c(x,y) \, \pi(dx,dy) + H(\pi|\mu\otimes\nu)
\end{equation}
is finite, there is a unique solution $\pi_{*}\in\Pi(\mu,\nu)$, and~$\pi_{*}$ is uniquely characterized within $\Pi(\mu,\nu)$ by having a density of the form 
  $$
    \frac{d\pi_{*}}{d(\mu\otimes\nu)} = e^{f\oplus g -c }\quad \mu\otimes\nu\as
  $$
for some measurable functions $f: \X\to\R$ and $g: \Y\to\R$, where we write $f\oplus g$ for $(x,y)\mapsto f(x)+g(y)$. The functions $f,g$ are called  called \emph{(Schr\"odinger) potentials} and they are integrable (under~$\mu$ and~$\nu$, respectively) in the cases relevant below; a sufficient condition for integrability is~$c\in L^{1}(\mu\otimes\nu)$. The sum $f\oplus g$ is uniquely determined $\mu\otimes\nu$-a.s., whereas the individual functions are unique up to an additive constant: clearly $(f+a,g-a)$ have the same sum for any $a\in\R$. Thus, they are unique after choosing a normalization removing this degree of freedom. See Appendix~\ref{se:background} for the preceding facts and further background.

Our aim is to establish stability of the potentials with respect to the marginals. Consider sequences $\mu_{n}\to \mu$ and $\nu_{n}\to\nu$ of marginals converging weakly (i.e., in the topology induced by bounded  continuous functions). Denoting by $(f_{n},g_{n})$ associated potentials, we want to state that $(f_{n},g_{n})\to (f,g)$ in a suitable sense. In view of the above characterization, a key step in this endeavor is to establish a form of compactness.
In general, it is not straightforward how to formalize the convergence of potentials. E.g., in a computational context, we may be interested in discrete measures $(\mu_{n},\nu_{n})$ approximating a continuous pair $(\mu,\nu)$. Then, these measures are  mutually singular and the spaces $L^{p}(\mu_{n})$ and $L^{p}(\mu)$ are not immediately comparable.  If $\mu\ll\mu_{n}$ (and similarly for $\nu_{n}$), this issue is milder as convergence in $\mu$-probability yields a natural topology. However, compactness still turns out to be an issue in the regime of interest. If~$c$ has high integrability, the weak-star topology in $L^{1}$ can be used, similarly to the arguments for Sinkhorn convergence in~\cite{Ruschendorf.95}. In some cases one can even use the Arzel\`a--Ascoli theorem (see Appendix~\ref{se:unifStability}).  But in the regime of interest here, where we want to cover quadratic cost and subgaussian marginals, we have not succeeded with off-the-shelf compactness concepts. 

Instead, we shall build compactness through an approximation scheme  and properties specific to the problem at hand, eventually using compactness of bounded sets in Euclidean space. This construction is the main technical contribution of the paper. It will also allow us to cover the case of mutually singular measures (in fact, focusing on equivalent measures would not result in a substantial simplification).
The approximation scheme has the form
\[
\begin{tikzcd}[column sep=3em]
F_{n}^{k} \arrow{r}{n\to\infty } \arrow[swap]{d}{k\to\infty} & F^{k} \arrow{d}{k\to\infty } \\%
f_{n} \arrow[dashed]{r} & f
\end{tikzcd}
\]
where for fixed~$n$, the potential $f_{n}$ is approximated by a sequence $(F_{n}^{k})_{k\in\N}$ converging in $\mu_{n}$-probability, with some additional uniformity in~$n$. The functions $F_{n}^{k}$ are piecewise constant;  more precisely, they are  simple functions based on a partition $(D^{k}_{j})_{j\in\N}$ of~$\X$ and have nonzero values on finitely many sets $D^{k}_{j}$. Using specific regularity properties of the potentials, these sets can be chosen independently of~$n$. Therefore, $(F_{n}^{k})_{n\in\N}$ can be identified with a sequence in a finite-dimensional space and compactness can be established from a priori bounds. Passing to a subsequence, this results in the (uniform) convergence $F_{n}^{k}\to F^{k}$ at the top of the diagram. The limiting functions $F^{k}$, in turn, are shown to form a Cauchy sequence in $\mu$-probability, thus yielding a limit~$f$ that is well-defined under~$\mu$. A similar construction is applied to $(g_{n})$, yielding a function $g$, and we shall prove that~$(f,g)$ are indeed potentials for the limiting marginals $(\mu,\nu)$.

The scheme in the diagram also acts as a way to formalize a strong convergence $f_{n}\to f$. It implies
convergence in distribution; that is, $(f_{n})_{\#}\mu_{n}\to f_{\#}\mu$ weakly, where $f_{\#}$ denotes the pushforward under $f$. Convergence in distribution is a natural notion given the weak convergence setting, but it is far from strong enough to imply the desired conclusions. If $\mu\ll\mu_{n}$, we show that our scheme implies the convergence in $\mu$-probability (and similarly for $\nu$) under a fairly general condition on the Radon--Nikodym derivatives; cf.\ Corollary~\ref{cor:abs_cont}. This condition is satisfied in particular whenever the marginals convergence in total variation, allowing us to deduce via Scheff\'e's lemma a result of its own interest (Corollary~\ref{co:TVstability}): the optimal couplings are stable in total variation; i.e., for marginals with 
$\|\mu_{n}-\mu\|_{TV}\to0$ and $\|\nu_{n}-\nu\|_{TV}\to0$, the corresponding  optimizers satisfy $\|\pi_{n}-\pi_{*}\|_{TV}\to0$.
Returning to the convergence of Sinkhorn's algorithm, we interpret each iteration of the algorithm as solving an entropic optimal transport problem with changing marginals $(\mu_{n},\nu_{n})$. These marginals converge in total variation to $(\mu,\nu)$, and we infer the convergence of the algorithm to the desired limit~$\pi_{*}$.

Several recent works have addressed the stability of entropic optimal transport from different angles. The first result is in~\cite{CarlierLaborde.20}, for a setting with bounded cost and marginals equivalent to a common reference measure with densities uniformly bounded above and below. The authors show by a differential approach that the potentials are continuous  in~$L^{p}$ relative to the marginal densities. %
Still with bounded cost (and some other conditions), \cite{DeligiannidisDeBortoliDoucet.21} establishes uniform continuity of the potentials relative to the marginals in Wasserstein distance~$W_{1}$; this result is based on the Hilbert--Birkhoff projective metric. 
Closer to the present unbounded setting, \cite{GhosalNutzBernton.21b} obtains stability of the optimal couplings in weak convergence for general continuous costs. Based on the geometric approach first proposed in~\cite{BerntonGhosalNutz.21}, the main restriction of the technique is that the underlying spaces need  to satisfy Lebesgue's theorem on differentiation of measures which generally holds only in finite-dimensional spaces. 
As a by-product, the main result of the present paper yields a similar stability result for weak convergence; cf.\ Theorem~\ref{thm:main}\,(i). The present result also applies  in an infinite-dimensional context; the more important difference, however,  is that we achieve a strong form of convergence, whereas~\cite{GhosalNutzBernton.21b} is silent about any convergence of the densities or potentials. In particular, we can infer stability in the sense of total variation convergence (Corollary~\ref{co:TVstability}) and the  corresponding convergence of Sinkhorn's algorithm  (Corollary~\ref{co:sink}). 
It is worth noting  that these two stabilities are at opposites ends of the spectrum: in the weak topology, compactness for sets of couplings is immediate; the difficulty  is to ensure that a  limit is the optimal coupling  for its marginals. For a limit in total variation, the latter is easy, but obtaining compactness is difficult due to the strength of the topology. In the present work, we effectively reduce the dimension by focusing on densities with a decomposition given by potentials and then obtain compactness through the potentials.
A last related work is~\cite{EcksteinNutz.21}, which was conducted concurrently. Here stability of the coupling in Wasserstein distance~$W_{p}$ is shown  under certain growth and integrability conditions. Obtained by control-theoretic arguments through a transport inequality, the main strength of this result lies in being quantitative (which the present one is not). On the other hand, \cite{EcksteinNutz.21} is once again silent about the densities or potentials, and does not yield a convergence in total variation. Indeed, we are not aware of previous stability results in total variation beyond bounded settings. Finally, we would like to mention the ongoing research~\cite{ChiariniConfortiGrecoTamanini.21} kindly pointed out to us by Giovanni Conforti. In the setting of dynamic Schr\"odinger bridges satisfying a logarithmic Sobolev inequality for the underlying dynamics and marginal distributions with finite Fisher information, the authors study quantitative bounds for the relative entropy of Schr\"odinger bridges with different marginals and the convergence of the gradients of the potentials towards the Brenier map as~$\eps\to0$.

The remainder of this paper is organized as follows. Section~\ref{se:mainResults} contains the main results on stability. Section~\ref{se:sink} details the application to Sinkhorn's algorithm. The proof of the main result, Theorem~\ref{thm:main}, is split into ten~steps which are reported in Section~\ref{se:mainProof}. For convenience, Appendix~\ref{se:background} summarizes background on entropic optimal transport. Appendix~\ref{se:unifStability} details how stability, even uniformly on compacts, can be obtained rather directly under strong integrability conditions. Lastly, Appendix~\ref{se:omittedProofs} contains some proofs that we defer in the body of the text.

\section{Stability}\label{se:mainResults}

Let $\X,\Y$ be Polish spaces endowed with their Borel $\sigma$-fields and $\cP(\X),\cP(\Y)$  their sets of Borel probability measures. We recall that $c:\X\times \Y\to [0,\infty)$ is continuous. In Theorem~\ref{thm:main} below, we consider the entropic optimal transport problem~\eqref{eq:EOT} for marginals $(\mu_{n},\nu_{n})\in\cP(\X)\times\cP(\Y)$ converging to marginals $(\mu,\nu)$. The condition~\eqref{eq:posBddL1Cond} of the theorem implies that $\cC(\mu_{n},\nu_{n})<\infty$ and that there exist optimal couplings $\pi_{n}\in\Pi(\mu_{n},\nu_{n})$ with associated potentials $(f _{n},g_{n})$;  see Section~\ref{se:background} for these facts and further background.

Before stating the theorem, let us comment on the normalization chosen therein.
As mentioned in the Introduction, $f _{n}$ and $g_{n}$ are only unique up to an additive constant. This does not affect the sum $f_{n}\oplus g_{n}$ determining the density of~$\pi_{n}$, but in order to obtain a separate convergence for $f_{n}$ and $g_{n}$, it is clearly necessarily to impose an additional condition to  pin down this constant. There are many possible choices; in Theorem~\ref{thm:main}, we work with $\alpha_{n}:=\int \arctan(f_n)\,d\mu_n$. As $\arctan$ is strictly increasing, fixing the value of the integral is equivalent to determining the additive constant (and conversely, there is a version of the potentials such that, e.g., $\alpha_{n}=0$). Since $\arctan$ is bounded, it is clear that $(\alpha_{n})$ always converges after passing to subsequence. Furthermore, this type of normalization is compatible with convergence in distribution. %

\begin{theorem}\label{thm:main}
For $n\in \N$, let $(\mu_n, \nu_n)\in \mathcal{P}(\mathcal{X})\times \mathcal{P}(\mathcal{Y})$ satisfy $\cC(\mu_{n},\nu_{n})<\infty$ and let $(f_{n},g_{n})$ be corresponding potentials. Suppose that %
\begin{align}\label{eq:posBddL1Cond}
\sup_{n\in\N} \int f^+_n\,d\mu_n<\infty, \qquad \sup_{n\in\N}\int g^+_n\,d\nu_n<\infty 
\end{align}
and that $\mu_n,\nu_n$ converge weakly to~$\mu,\nu$. Then
\begin{enumerate}[labelindent=\parindent, labelindent=0pt, leftmargin=*, label=(\roman*), 
align=left,  topsep=6pt, itemsep=6pt]
\item[(i)] 
$\cC(\mu,\nu)<\infty$ and the optimal couplings converge weakly: $\pi_{n}\to\pi_{*}$.
\end{enumerate}
\begin{enumerate}[labelindent=\parindent, labelindent=0pt, leftmargin=*, label=(\roman*), 
align=left,  topsep=6pt, itemsep=6pt]
\item[(ii)]
Suppose also that the normalizations $\alpha_{n}:=\int \arctan(f_n)\,d\mu_n$ converge to a limit~$\alpha$ (this always holds along a subsequence).
Let $(\eps_k)_{k\in\N}\subset (0,1)$ satisfy $\eps_k \downarrow0$. There are measurable functions 
$$
F_{n}^{k},F^{k},f:\X\to\R,  \qquad G_{n}^{k},G^{k},g:\Y\to\R, \qquad k,n\in\N
$$
and a subsequence $(n_{l})_{l\in\N}\subset\N$ such that
\begin{align}
&\mbox{$(f,g)\in L^{1}(\mu)\times L^{1}(\nu)$ are the potentials for $(\mu,\nu)$ with $\int \arctan(f)\,d\mu=\alpha$}, \label{eq:limitingPotentials}\\
&\mbox{$(f_{n},\mu_{n}) \to (f,\mu)$ in distribution; that is, $(f_{n})_{\#}\mu_{n}\to f_{\#}\mu$,} \label{eq:claimConvInDistrib}\\
&\lim_{k\to\infty} \sup_{n\in \N} \mu_n( |f_n-F^k_n|\ge \eps_{k})=0, \label{eq:claim2} \\
& \lim_{l\to\infty} F_{n_{l}}^k =F^k \; \mbox{ uniformly, \; for each $k\in\N$,} \label{eq:claim3} \\
&\lim_{k\to\infty} F^k = f \;\mbox{ in $\mu$-probability,}  \label{eq:claim1}
\end{align}
and analogous properties hold for $G_{n}^{k},G^{k},g$. 

\item[(iii)] 
Suppose that $(f^{+}_{n},g^{+}_{n})_{n\in\N}$ are uniformly integrable wrt.\ $(\mu_{n},\nu_{n})_{n\in\N}$\textup{;} that is,
\begin{align}\label{eq:uniform_int}
\lim_{C\to \infty}\sup_{n\in \N} \int f_n \1_{f_n>C}\,d\mu_n=0, \qquad \lim_{C\to \infty}\sup_{n\in \N} \int g_n \1_{g_n>C}\,d\nu_n=0.
\end{align}
Then the optimal values converge: $\cC(\mu_{n},\nu_{n})\to\cC(\mu,\nu)$. If $\alpha_{n},\alpha,f,g$ are as in~(ii), then 
\begin{align}\label{eq:meansConverge}
   \lim_{n\to\infty} \int f_n\,d\mu_n =\int f\,d\mu, \qquad \lim_{n\to\infty} \int g_n\,d\nu_n =\int g\,d\nu.
\end{align}
\end{enumerate}
\end{theorem}

The next two results discuss the main condition~\eqref{eq:posBddL1Cond} of the theorem.

\begin{remark}\label{rk:bddL1Cond}
Condition~\eqref{eq:posBddL1Cond} implies 
\begin{align}\label{eq:negBddL1Cond}
 \inf_{n\in\N} \int f_n\,d\mu_n>-\infty, \qquad \inf_{n\in\N} \int g_n\,d\nu_n>-\infty,
\end{align}
and thus~\eqref{eq:posBddL1Cond} is equivalent to boundedness in $L^{1}$:
\begin{align}\label{eq:bddL1Cond}
\sup_{n\in\N} \int |f_n|\,d\mu_n<\infty, \qquad \sup_{n\in\N}\int |g_n|\,d\nu_n <\infty.
\end{align}
\end{remark}

Remark~\ref{rk:bddL1Cond} follows from the duality $\int f_n\,d\mu_n+\int g_n\,d\nu_n=\cC(\mu_{n},\nu_{n})\geq0$; cf.\ Proposition~\ref{pr:duality}.
We can provide a sufficient condition for~\eqref{eq:posBddL1Cond} in terms of the given data as follows.

\begin{lemma}\phantomsection\label{le:suff_cond_for_bddL1} 
\begin{enumerate}[labelindent=\parindent, labelindent=0pt, leftmargin=*, label=(\roman*), 
align=left,  topsep=6pt, itemsep=6pt]
\item[(i)] Let $(f_{n},g_{n})$ satisfy~\eqref{eq:negBddL1Cond}. The following condition is sufficient for~\eqref{eq:posBddL1Cond} and~\eqref{eq:bddL1Cond}\textup{:}
\begin{align}\label{eq:condition_cost_int}
\sup_{n\in\N} \int c\,d (\mu_n\otimes \nu_n)<\infty.
\end{align}
If $c$ is uniformly integrable wrt.\ $(\mu_n\otimes \nu_n)_{n\in\N}$, then~\eqref{eq:uniform_int} holds as well.

\item[(ii)] The following condition is sufficient for~\eqref{eq:condition_cost_int}\textup{:} 
\begin{align}\label{eq:condition_entropy}
\sup_{n\in \N} \big[H(\mu_{n}|\mu)+H(\nu_{n}|\nu)\big]<\infty \qandq \mbox{$e^{\beta c}\in L^{1}(\mu\otimes\nu)$ \;for some $\beta>0$}.
\end{align}
If moreover $e^{\beta \phi(c)}\in L^{1}(\mu\otimes\nu)$ for an increasing, superlinearly growing function~$\phi$, then $c$ is uniformly integrable wrt.\ $(\mu_n\otimes \nu_n)_{n\in\N}$.
\end{enumerate}
\end{lemma}

The proof is deferred to Appendix~\ref{se:omittedProofs}. 
We remark that the assumption~\eqref{eq:negBddL1Cond} is mostly a matter of normalization. Indeed, suppose that $f_{n}^{+}\in L^{1}(\mu_{n})$ and $g_{n}^{+}\in L^{1}(\nu_{n})$---which necessarily holds under~\eqref{eq:condition_cost_int}, cf.\ Proposition~\ref{pr:existence}. Then $\int f_n\,d\mu_n+\int g_n\,d\nu_n=\cC(\mu_{n},\nu_{n})\geq0$ by duality (cf.\ Proposition~\ref{pr:duality}). Therefore, \eqref{eq:negBddL1Cond} always holds after choosing a suitable normalization for $(f_{n},g_{n})$, for instance the centering $\int f_n \,d\mu_n=0$.

In the remainder of the section, we discuss corollaries of Theorem~\ref{thm:main} that hold when the marginals $(\mu_{n},\nu_{n})$  have additional properties.
In the general setting of Theorem~\ref{thm:main}, there is no natural function space where one could formulate the convergence $f_{n}\to f$ in a straightforward way. Whereas if $\mu\ll \mu_n$ for all $n\in \N$, the potentials $(f_{n})$ are well-defined $\mu$-a.s.\ and one can naturally ask whether $f_{n}\to f$ in $\mu$-probability. The following gives an affirmative answer under a  weak condition of boundedness-in-probability on the Radon--Nikodym derivatives.

\begin{corollary}\label{cor:abs_cont}
Let~\eqref{eq:posBddL1Cond} hold and let $\mu_n,\nu_n$ converge weakly to~$\mu,\nu$. Suppose that $\mu\ll \mu_n$, $\nu\ll \nu_n$  and
\begin{align}\label{eq:bounded_prob}
 \limsup_{C\to\infty}\limsup_{n\to\infty} \mu\left( \frac{d\mu}{d\mu_n}\ge C\right) =0, \qquad
 \limsup_{C\to\infty}\limsup_{n\to\infty} \nu\left( \frac{d\nu}{d\nu_n}\ge C\right) =0.
\end{align}
Then $f_{n}\oplus g_{n}\to f\oplus g$ in $\mu\otimes\nu$-probability, where $(f,g)$ are arbitrary potentials for~$(\mu,\nu)$. If $\alpha_{n},\alpha,f,g$ are as in~Theorem~\ref{thm:main}\,(ii), then $f_n\to f$ in $\mu$-probability and $g_n\to g$ in $\nu$-probability.
\end{corollary}

\begin{proof}
 As $f\oplus g$ is uniquely determined and convergence in probability is metrizable, it suffices to show that any subsequence of $f_{n}\oplus g_{n}$ has a subsequence converging to~$f\oplus g$. Thus we may further assume that $\alpha_{n},\alpha,f,g$ are as in~Theorem~\ref{thm:main}\,(ii) and show the convergence of~$(f_n)$~and $(g_{n})$. Fix $\eps>0$ and a subsequence of $(f_n)$. By~\eqref{eq:claim1} and~\eqref{eq:claim3} in Theorem~\ref{thm:main}, there exist a further subsequence (not relabeled) and functions $F_{n}^k, F^k$ such that along this subsequence,
\begin{align}\label{eq:claim_n1}
\lim_{k\to \infty} \mu(|F^k-f|\ge \eps)&=0 \quad\qandq\quad \lim_{n\to \infty} \mu ( |F_n^k -F^k|\ge \eps)=0.
\end{align}
On the  other hand, for any $C>0$,
\begin{align*}
\mu( |f_n-F^k_n|\ge \eps)
& \leq \mu\left(  \frac{d\mu}{d\mu_n}\ge C \right) + \int \1_{\{ |f_n-F^k_n|\ge \eps\}\cap\{\frac{d\mu}{d\mu_n}\leq C\}} \frac{d\mu}{d\mu_n}\,d\mu_n\\
& \leq \mu\left(  \frac{d\mu}{d\mu_n}\ge C \right) + C\mu_n( |f_n-F^k_n|\ge \eps) 
\leq \mu\left(  \frac{d\mu}{d\mu_n}\ge C \right) + C \delta_{k}
\end{align*}
for some $\delta_{k}\geq0$ with $\lim_{k}\delta_{k}=0$, by~\eqref{eq:claim2}. %
Taking $n\to \infty$,  then $k\to \infty$ and finally $C\to \infty$, we obtain 
\begin{align*}
\limsup_{k\to\infty}\limsup_{n\to \infty} \mu( |f_n-F^k_n|\ge \eps)=0.
\end{align*}
Together with~\eqref{eq:claim_n1}, this yields
\begin{align*}
\lim_{n\to \infty} \mu(|f_n-f|\ge 3\eps)&\le  \limsup_{k\to \infty} \limsup_{n\to \infty} \mu( |f_n-F^k_n|\ge \eps ) \\
& \quad + \limsup_{k\to \infty} \limsup_{n\to \infty} \mu (  |F_n^k -F^k|\ge \eps)  \\
& \quad +\limsup_{k\to \infty} \mu(|F^k-f|\ge \eps)=0.
\end{align*}
The proof for $g$ is analogous.
\end{proof}

Condition~\eqref{eq:bounded_prob} holds in particular for sequences converging in total variation.

\begin{lemma}\label{le:TVconvImpliesCond}
  Suppose that $\mu\ll \mu_n$ and $\mu_{n}\to\mu$ in total variation. Then
$$
  \limsup_{C\to\infty}\limsup_{n\to\infty} \mu\left( \frac{d\mu}{d\mu_n}\ge C\right) =0.
$$
\end{lemma} 

The proof is deferred to Appendix~\ref{se:omittedProofs}. 
Our final result is the stability of the optimal couplings in the topology of total variation, complementing the weak stability shown in Theorem~\ref{thm:main}\,(i).

\begin{corollary}\label{co:TVstability}
Let~\eqref{eq:posBddL1Cond} hold and let $\mu_n,\nu_n$ converge in total variation to~$\mu,\nu$ where $\mu\ll\mu_{n}$ and $\nu\ll\nu_{n}$. Then $\pi_{n}\to\pi_{*}$ in total variation.
\end{corollary}

\begin{proof}
 For the sake of readability, we state here the proof under the additional assumption that $\mu_n\sim \mu$ and $\nu_n\sim \nu$; the general case is deferred to Appendix~\ref{se:omittedProofs}. 
 By Corollary~\ref{cor:abs_cont} and Lemma~\ref{le:TVconvImpliesCond}, we have $f_{n}\to f$ in $\mu$-probability and $g_{n}\to g$ in $\nu$-probability after passing to a subsequence.
 Under the additional assumption, $\frac{d\mu_n}{d\mu}\to1$ in $L^{1}(\mu)$ and $\frac{d\nu_n}{d\nu}\to1$ in $L^{1}(\nu)$. We see that
 $$
   \frac{d\pi_{n}}{d(\mu\otimes\nu)} = \frac{d\pi_{n}}{d(\mu_{n}\otimes\nu_{n})}\frac{d\mu_n}{d\mu}\frac{d\nu_n}{d\nu} = e^{f_{n}\oplus  g_{n}-c}\frac{d\mu_n}{d\mu}\frac{d\nu_n}{d\nu} 
   \,\to\, e^{f\oplus  g-c} = \frac{d\pi_{*}}{d(\mu\otimes\nu)}
 $$
 in $\mu\otimes\nu$-probability. But then the convergence also holds in $L^{1}(\mu\otimes\nu)$, by  Scheff\'e's lemma, and we conclude that $\pi_{n}\to\pi_{*}$ in total variation. The convergence of the original sequence follows.
\end{proof}

\section{Convergence of Sinkhorn's Algorithm}\label{se:sink}

Fix marginals $(\mu,\nu)\in\cP(\X)\times\cP(\Y)$ and a continuous cost $c:\X\times\Y\to[0,\infty)$. Sinkhorn's algorithm~\eqref{eq:SinkPrimalImplicit} can be written in terms of potentials. Set $\varphi_{0}:=0$  and 
\begin{align}\label{eq:SinkDual}
\begin{split}
    \psi_{t}(y) &:= -\log \int_{\X} e^{\varphi_{t}(x)-c(x,y)} \, \mu(dx),\\
    \varphi_{t+1}(x)&:= -\log \int_{\Y} e^{\psi_{t}(y)-c(x,y)}\, \nu(dy)
\end{split}  
\end{align}
for $t\geq0$, and
define the measures
$$
  d\pi(\varphi,\psi) :=e^{\varphi\oplus\psi-c}\,d(\mu\otimes\nu), \qquad 
\pi_{2t}:=\pi(\varphi_{t},\psi_{t}), \qquad \pi_{2t-1} := \pi(\varphi_{t},\psi_{t-1}),
$$
where $\psi_{-1}:=0$. One can check by direct calculation that~$\pi_{n}\in\cP(\X\times\Y)$ are the same measures as in~\eqref{eq:SinkPrimalImplicit}. Denoting by $(\mu_{n},\nu_{n})$ the marginal distributions of~$\pi_{n}$, the following summarizes well known properties of Sinkhorn's algorithm (e.g., \cite[Section~6]{Nutz.20}).

\begin{lemma}\label{le:sinkFacts}
  Let $\cC(\mu,\nu)<\infty$. We have $\mu_{n}\sim\mu$ and $\nu_{n}\sim\nu$ for all $n\geq0$. Moreover, 
  $H(\mu_{n}|\mu)+H(\nu_{n}|\nu)\to0$; in  particular, $\mu_{n}\to\mu$ and $\nu_{n}\to\nu$ in total variation. For $t\geq0$, the marginals satisfy
\begin{equation}\label{eq:sinkItMarginals}
  \mu_{2t+1}=\mu, \qquad \nu_{2t}=\nu, \qquad \frac{d\mu_{2t}}{d\mu}=e^{\varphi_{t}-\varphi_{t+1}}, \qquad \frac{d\nu_{2t-1}}{d\nu}=e^{\psi_{t-1}-\psi_{t}}.
\end{equation}  
  It follows that for $n\geq1$,
  $$
   \frac{d\pi_{n}}{d(\mu_{n}\otimes\nu_{n})}
  =e^{f_{n}\oplus g_{n}-c} \quad \mu_{n}\otimes\nu_{n}\as,
  $$   
  where
  \begin{equation}\label{eq:sinkToPotentials}
  \begin{cases}
    f_{n} := \varphi_{t+1}, \quad& g_{n} := \psi_{t}\qquad \mbox{ if }~n=2t, \\
    f_{n} := \varphi_{t}, \quad &g_{n} := \psi_{t}\qquad \mbox{ if }~n=2t-1. 
  \end{cases}  
  \end{equation}
\end{lemma}

In brief, $(f_{n},g_{n})$ are potentials for the marginals $(\mu_{n},\nu_{n})$ which in turn converge to~$(\mu,\nu)$ in total variation. The stability result of Corollary~\ref{co:TVstability} then yields the following  convergence result. As  emphasized in the Introduction, it covers quadratic costs with arbitrary subgaussian marginals and the problem~\eqref{eq:epsOT} with arbitrary regularization parameter~$\eps>0$.

\begin{corollary}\label{co:sink}
  Suppose that 
  \begin{equation*}%
    \mbox{$e^{\beta c}\in L^{1}(\mu\otimes\nu)$ \;for some $\beta>0$}.
  \end{equation*}
  Then $\cC(\mu,\nu)<\infty$ and the Sinkhorn iterates $(\pi_{n})$ converge to~$\pi_{*}$ in total variation.
\end{corollary}

\begin{proof}
  As $H(\mu_{n}|\mu)+H(\nu_{n}|\nu)\to0$ by Lemma~\ref{le:sinkFacts}, Lemma~\ref{le:suff_cond_for_bddL1} yields that%
  \begin{align}\label{eq:SinkPrf1}
    \sup_{n\in\N} \int c\,d (\mu_n\otimes \nu_n)<\infty.
  \end{align} 
  In particular, $\cC(\mu_{n},\nu_{n})<\infty$ and $(f_{n},g_{n})$ as defined in Lemma~\ref{le:sinkFacts} are potentials for the marginals $(\mu_{n},\nu_{n})$; cf.\ Propositions~\ref{pr:existence} and~\ref{pr:duality}. 
Next, we show that~$(f_{n},g_{n})$ satisfy~\eqref{eq:posBddL1Cond}. In general, the Sinkhorn iterates satisfy $(\varphi_{t},\psi_{t})\in  L^{1}(\mu)\times L^{1}(\nu)$ as well as $\int \varphi_{t}\,d\mu\geq0$ and $\int \psi_{t}\,d\nu\geq -\log \int e^{-c}\,d(\mu\otimes\nu)$; see \cite[Lemma~6.4 and its footnote]{Nutz.20}. Here, as $c\geq0$, we have $\int \varphi_{t}\,d\mu\geq0$ and $\int \psi_{t}\,d\nu\geq0$.

Consider~$n=2t$, then $\nu_{n}=\nu$. Using~\eqref{eq:SinkDual}, Jensen's inequality and $\int \psi_{t}\,d\nu\geq0$,
  \begin{align*}
    f_{n}(x) = \varphi_{t+1}(x) \leq  \int c(x,y) \, \nu(dy) =  \int c(x,y) \, \nu_{n}(dy)
  \end{align*}
  and hence 
  $
    \int f_{n}^{+}\,d\mu_{n} \leq \int c \, d(\mu_{n}\otimes\nu_{n})
  $
  which is bounded by~\eqref{eq:SinkPrf1}. Similarly, \eqref{eq:SinkDual} implies
  $
    g_{n}(y) =  \psi_{t}(y) \leq \int c(x,y) \, \mu(dx)
  $
  and hence 
  \begin{align*}
    \int g_{n}^{+}\,d\nu_{n} = \int g_{n}^{+}\,d\nu \leq \int c \, d(\mu\otimes\nu) <\infty.
  \end{align*}  
  The argument for~$n=2t-1$ is symmetric, so that~\eqref{eq:posBddL1Cond} holds.
  As the marginals are  equivalent and converge in total variation by Lemma~\ref{le:sinkFacts}, the claim follows by Corollary~\ref{co:TVstability}. 
\end{proof} 

\section{Proof of Theorem \ref{thm:main}}\label{se:mainProof}

The proof is structured into several steps.
\begin{itemize}[labelindent=3em, leftmargin=*,align=left,  topsep=9pt, itemsep=2pt]
\item[Step~\ref{st:step1}:] Definition of sets $A_{n}^{k}, \Xcpt^{k}$.
\item[Step~\ref{st:step2}:] Definition of $f_{n}^{k}$.
\item[Step~\ref{st:step3}:] Boundedness of $(f_{n}^{k})_{n\in\N}$ on compact sets.
\item[Step~\ref{st:step4}:] Construction of $F_{n}^{k}$  and $F^{k}$.
\item[Step~\ref{st:step5}:] Further properties related to $F_{n}^{k}$  and $F^{k}$.
\item[Step~\ref{st:step6}:] Proof that  $(F^{k})_{k\in\N}$ is $\mu$-Cauchy and definition of $f$.
\item[Step~\ref{st:step7}:] Proof that $(f_{n})_{\#}\mu_{n}\to f_{\#}\mu$.
\item[Step~\ref{st:step8}:] Proof that $f,g$ induce a coupling $\pi$ and $\pi_{n}\to\pi$.
\item[Step~\ref{st:step9}:] Identification of the limit, end of proof of Theorem~\ref{thm:main}\,(i),(ii).
\item[Step~\ref{st:step10}:] Proof of Theorem~\ref{thm:main}\,(iii).
\end{itemize} 

As noted in Appendix~\ref{se:background}, we may choose versions of the potentials  satisfying the Schr\"odinger equations~\eqref{eq:SE} everywhere. For brevity, we generally only detail the arguments for~$(f_n)$, the arguments for~$(g_n)$ are symmetric. Recall that a sequence $(\eps_k)_{k\in\N}\subset (0,1)$ with $\eps_{k}\downarrow 0$ is given. We define $(\tilde{\delta}_k)_{k\in\N}\subset(0,1/2)$ by $\tilde{\delta}_k=\eps_{k}/2$ and another sequence $(\delta_{k})_{k\in\N}$ by $\delta_k= 1-e^{-\tilde{\delta}_k}$; i.e.,
\begin{equation}\label{eq:deltaRel}
  \tilde{\delta}_k= -\log(1-\delta_k).
\end{equation}
It follows that $0\leq \delta_k\leq \tilde{\delta}_k$ and $\tilde{\delta}_{k}\downarrow 0$. For notational convenience, we set $\mu_0:=\mu$, $\nu_0:=\nu$ and $\N_0:=\N\cup \{0\}$.

Step~\ref{st:step1} is based on the following generalization of~\cite[Lemma~2.3]{NutzWiesel.21}  extending that result from a single measure to a tight set of measures.

\begin{lemma}\label{lem:ext}
Let $\delta\in (0,1)$. There are compact sets $\Xcpt(\delta)\subseteq \mathcal{X}$, $\Ycpt(\delta)\subseteq \mathcal{Y}$ with 
$$ 
\mu_n(\Xcpt(\delta))\ge 1-\delta, \qquad \nu_n(\Ycpt(\delta))\ge 1-\delta, \qquad n\in \N_0
$$
and measurable sets $A_{n}\subseteq \Xcpt(\delta)$, $B_{n}\subseteq \Ycpt(\delta)$ for $n\in\N$ such that 
\begin{align*}
&\mu_n(A_{n})\ge 1-\delta, \qquad \nu_n(B_{n}) \ge 1-\delta, \qquad n\in \N, \\[.3em]
\left|f_n(x_1)-f_n(x_2)\right|&\le \sup_{y\in \Ycpt(\delta)} \left|c(x_1, y)-c(x_2,y)\right| -\log(1-\delta) \quad\mbox{for} \quad x_1,x_2\in A_{n},\\[.5em]
\left|g_n(y_1)-g_n(y_2)\right|%
&\le \sup_{x\in \Xcpt(\delta)} \left|c(x, y_1)-c(x,y_2)\right| -\log(1-\delta)  \quad\mbox{for} \quad y_1, y_2\in B_{n}.
\end{align*}
\end{lemma}

The proof of Lemma~\ref{lem:ext} is an adaptation of the arguments in~\cite{NutzWiesel.21}; for completeness, the details are reported in Appendix~\ref{se:omittedProofs}.

\begin{step}\label{st:step1}
Fix $k\in \N$.  Lemma~\ref{lem:ext} with $\delta=\delta_{k}$ yields sets
$$
A_{n}^{k}\subseteq \Xcpt^{k}\subseteq \mathcal{X}\qquad \text{and}\qquad B_{n}^{k}\subseteq \Ycpt^{k}\subseteq \mathcal{Y}
$$
such that 
$$
\mu_n(A_{n}^{k})\ge 1-\delta_k,\qquad \nu_n (B_{n}^{k})\ge 1-\delta_k, \qquad n\in\N;
$$
moreover, $\Xcpt^{k},\Ycpt^{k} $ are compact and 
$$
 \mu_n(\Xcpt^{k})\ge 1-\delta_k, \qquad \nu_n(\Ycpt^{k} )\ge 1-\delta_k, \qquad n\in\N_{0}.
$$
Recalling~\eqref{eq:deltaRel},  we have for all $n\in\N$ that
\begin{align}\label{eq:continuity}
\left|f_n(x_1)-f_n(x_2)\right|
&\le \sup_{y\in \Ycpt^{k}} \left|c(x_1, y)-c(x_2,y)\right| +\tilde{\delta}_k ,\qquad  x_1,x_2\in A_{n}^{k}
\end{align}
and similarly for $g_{n}$. 
\end{step}

\begin{step}\label{st:step2}
Define the continuous pseudometric $\tilde{d}_k$ on~$\mathcal{X}$ by
\begin{align*}
\tilde{d}_k(x_1,x_2):= \sup_{y\in \Ycpt^{k}} \left|c(x_1, y)-c(x_2,y)\right|.
\end{align*}
Let $n\in\N$. Using~\eqref{eq:continuity} and a version of Kirszbraun's extension theorem as detailed in \cite[Lemma~2.4]{NutzWiesel.21}, there exists a function $f^k_n: \X\to\R$ satisfying%
\begin{align}
 &f^k_n=f_n\quad\mbox{on}\quad  A_{n}^{k}, \nonumber\\
 &|f^k_n(x_1)-f^k_n(x_2) |\le \tilde{d}_k(x_1,x_2) +\tilde{\delta}_k, \quad x_1,x_2\in \mathcal{X}. \label{eq:extension1}
\end{align}
\end{step}

\begin{step}\label{st:step3}
With $k\in\N$ still fixed, we show that $(f_n^k)_{n\in \N}$ is bounded on compact sets. Suppose that $\sup_{n\in\N}\sup_{x\in K} |f_n^k(x)|=\infty$  for some compact  $K\subset\X$. Then~\eqref{eq:extension1} even implies that 
\begin{equation}\label{eq:pointwiseBddForContrad}
  \sup_{n\in\N}  \inf_{x\in K'} |f_n^k(x)|=\infty \quad\mbox{for any compact set $\emptyset\neq K'\subset\X$},
\end{equation}
because $\sup_{x_{1}\in K,\,x_{2}\in K'}\tilde{d}_k(x_1,x_2)<\infty$. We shall contradict~\eqref{eq:pointwiseBddForContrad}. Indeed, by~\eqref{eq:bddL1Cond} and Markov's inequality there exists $C>0$ such that
$
\mu_n\left(|f_n|\ge  C\right)\le \delta_k
$
for all $n\in\N$. As $\mu_n(A_{n}^{k})\ge 1-\delta_k$, it follows that
\begin{align*}
\mu_n\big( \{|f_n|\le C \} \cap A_{n}^{k}\big)\ge 1-2\delta_k>0
\end{align*}
and in particular $\{|f_n|\le C\} \cap A_{n}^{k}\neq\emptyset$. As $f_n^k=f_n$ on $A_{n}^{k}$ and $A_{n}^{k}\subseteq \Xcpt^{k}$, it follows that $\{|f_n^k|\le C\} \cap \Xcpt^{k}\neq\emptyset$ for all $n\in\N$, contradicting~\eqref{eq:pointwiseBddForContrad} for $K':= \Xcpt^{k}$.
\end{step}

As a preparation for Step~\ref{st:step4}, we record the following covering lemma.

\begin{lemma}\label{lem:partition}
Let $K\subset\mathcal{X}$ be compact and $r>0$. There exists a measurable partition $(D_j)_{j\in \N}$ of $\mathcal{X}$ such that 
\begin{enumerate}
\item $D_j$ has diameter at most $r$ and boundary $\mu(\partial D_{j})=0$, %
\item $K$ intersects finitely many elements of $(D_j)_{j\in \N}$.
\end{enumerate}
\end{lemma}

\begin{proof}
Pick a dense sequence $(x_j)_{j\in \N}$ in $\mathcal{X}$. For each $j$, the boundaries $\partial B_{\rho}(x_j)$ are disjoint for different values of $\rho>0$. Hence there are at most countably many  values of $\rho$ such that $\mu(\partial B_{\rho}(x_j))>0$ for some~$j$, and we can pick $\rho\in(0,r/2]$ such that $\mu(\partial B_{\rho}(x_j))=0$ for all~$j$.

By compactness there is a finite subset $\cN\subset\N$ such that $B_{\rho}(x_j), j\in \cN$ cover~$K$; we may assume that $\cN=\{1,\dots,N\}$. Set $D_{0}:=\emptyset$ and 
$$
  D_{j} = B_{\rho}(x_j)\setminus (D_{1}\cup\cdots\cup D_{j-1}),\qquad j\geq1.
$$
The general relation $\partial (A\cap B)\subseteq \partial A \cup \partial B$ implies that $\mu(\partial D_{j})=0$, and the other requirements are satisfied by construction.
\end{proof}

\begin{step}\label{st:step4}
Keeping $k\in\N$ fixed, our next aim is to define the functions~$F_{n}^{k}$ and~$F^{k}$. As $\tilde{d}_k$ is uniformly continuous on the compact set~$\Xcpt^{k}$, there is $r\in(0,\tilde{\delta}_k)$ such that
\begin{align}\label{eq:new3}
\sup_{x_1,x_2 \in \Xcpt^{k},\ d(x_1, x_2)\le r} \tilde{d}_k(x_1, x_2)< \tilde{\delta}_k.
\end{align} 
We apply Lemma~\ref{lem:partition} with $K=\Xcpt^{k}$ to define a partition  $(D_{j}^k)_{j\in \N}$ of~$\X$  and choose a sequence $(x_{j,n}^{k})_{j,n\in \N}\subset\X$  satisfying 
\begin{align*}
x_{j,n}^{k} \in \begin{cases}
D_{j}^k\cap A_{n}^{k} & \text{if } D_{j}^k\cap A_{n}^{k}\neq \emptyset,\\
D_{j}^k\cap \Xcpt^{k} & \text{if }  D_{j}^k\cap A_{n}^{k}= \emptyset\text{ and } D_{j}^k\cap \Xcpt^{k}\neq \emptyset,\\
D_{j}^k &\text{otherwise}.
\end{cases}
\end{align*}
We then set
\begin{align*}
F_n^k(x):=\sum_{j\in \N} f_n^k(x_{j,n}^{k}) \1_{D_{j}^k\cap \Xcpt^{k}}(x).
\end{align*}
This is a finite sum as $\cJ^{k}:=\{j:\, D_{j}^k\cap \Xcpt^{k}\neq\emptyset\}$ is finite (cf.\ Lemma~\ref{lem:partition}).  Moreover, the points $x_{j,n}^{k}$ with $j\in\cJ^{k}$ all belong to the compact set~$\Xcpt^{k}$. In view of Step~\ref{st:step3}, it follows that the corresponding coefficients $f_n^k(x_{j,n}^{k})$ are bounded uniformly in~$j,n$.

We can now apply a diagonal sequence argument to extract a subsequence (not relabeled) along which $a_j^k:=\lim_{n\to\infty} f_n^k(x_{j,n}^{k})$ exists for all $j\in\N$. Set
\begin{align}\label{eq:defFk}
F^k(x):=\sum_{j\in \N} a_j^k \1_{D_{j}^k\cap \Xcpt^{k}}(x),
\end{align}
then $F_n^k \to F^k$ pointwise on $\mathcal{X}$ and in particular~$F^k$ is measurable. In fact, this convergence is uniform as~$\cJ^{k}$ is finite. By construction, $F_n^k,F^{k}$ are supported on the compact~$\Xcpt^{k}$ and bounded uniformly in~$n$. Passing to another subsequence, we achieve that $\lim_{n\to \infty} F_n^k=F^k$ holds simultaneously for all~$k\in \N$. %
For the remainder of the proof, we only consider this subsequence (denoted $(n_{l})_{l}$ in Theorem~\ref{thm:main} but not relabeled here).
\end{step}

\begin{step}\label{st:step5}
We record two more facts about the construction in Step~\ref{st:step4} for later use. For $x \in D_{j}^k\cap A_{n}^{k}$, \eqref{eq:continuity} and~\eqref{eq:new3} yield
\begin{align*}%
|F_n^k(x)-f_n^k(x)|=|f_n^k(x_{j,n}^{k})-f_n^k(x)| \le \tilde{d}_k(x_{j,n}^{k}, x) +\tilde{\delta}_k< 2\tilde{\delta}_k.
\end{align*}
As $\cup_{j}D_{j}^k=\X$ and $f_n^k=f_n$ on $A_{n}^{k}$, it follows that
\begin{align}\label{eq:new4}
|F_n^k(x)-f_n(x)|=|F_n^k(x)-f_n^k(x)|< 2\tilde{\delta}_k, \qquad x \in A_{n}^{k}.
\end{align}
In view of $\mu_n (A_{n}^{k})\ge 1-\delta_k$, this yields in particular 
$$
  \sup_{n\in \N} \mu_n(|f_n-F^k_n|\ge 2\tilde{\delta}_k)\le \delta_k
$$ 
and hence the claim~\eqref{eq:claim2};  recall $2\tilde{\delta}_k=\eps_{k}$.

Second, we define similarly as in~\eqref{eq:defFk} the function
\begin{align}\label{eq:defFkreg}
F^k_{\mathrm{reg}} (x):=\sum_{j\in \cJ^{k}} a_j^k \1_{D_{j}^k}(x).
\end{align}
Then $F^k_{\mathrm{reg}} =F^k$ on $\Xcpt^{k}$ and hence
\begin{equation}\label{eq:FkregCompare}
  \mu_{n}(F^k_{\mathrm{reg}} \neq F^k) \le \delta_{k}, \qquad n\in\N_{0}.
\end{equation}
Clearly $F^k_{\mathrm{reg}}$ is continuous on the interior of the complement of $U:=\cup_{j\in\cJ^{k}}\partial D_{j}^k$. As $\cJ^{k}$ is finite, $\partial D_{j}^k$ is closed and $\mu(\partial D_{j}^k)=0$, the set~$U$ is closed and $\mu$-null. Hence its complement is open and has full $\mu$-measure; in brief, $F^k_{\mathrm{reg}}$ is continuous $\mu$-a.s. In particular, $F^k_{\mathrm{reg}}$ can be used in the mapping theorem \cite[Theorem~2.7]{Billingsley.99} for weak convergence arguments with $\mu_{n}\to\mu$.
\end{step}

\begin{step}\label{st:step6}
We show that $(F^k)_{k\in \N}$ is Cauchy in $\mu$-probability. Fix $\eps>0$ and choose $k_0\in \N$ such that $\tilde{\delta}_{k} \le \eps$ (hence also $\delta_{k}\leq \eps$) for all $k\ge k_0$. Let $k, k'\ge k_0$. For $x\in A_{n}^{k}\cap  A_{n}^{k'}$ we have $f^k_n(x)=f_{n}(x)=f^{k'}_n(x)$ and thus ~\eqref{eq:new4} yields
\begin{align*}
|F^k_n(x)-F^{k'}_n(x)|&\le |F^k_n(x)-f^{k}_n(x)|+|f^k_n(x)-f^{k'}_n(x)|+|f^{k'}_n(x)-F^{k'}_n(x)|\\
&\le  2\tilde{\delta}_k + 0 + 2\tilde{\delta}_{k'} \leq 4\eps.
\end{align*}
As a result,
\begin{align*}%
|F^k(x)-F^{k'}(x)| &\le |F^k(x)-F^{k}_n(x)|+|F^k_n(x)-F^{k'}_n(x)|+|F^{k'}_n(x)-F^{k'}(x)|\\
&\le |F^k(x)-F^{k}_n(x)| + 4\eps +|F^{k'}_n(x)-F^{k'}(x)|
\end{align*}
for $x\in A_{n}^{k}\cap A_{n}^{k'}$. In view of $\mu_{n}(A_{n}^{k}\cap A_{n}^{k'})\geq 1-2\eps$, it follows that
\begin{align*}
\int  \left[ \left|F^k-F^{k'}\right|\wedge 1\right]\,d\mu_n 
\le \int \left[\left|F^k-F_n^k\right|\wedge 1+\left|F_n^{k'}-F^{k'}\right|\wedge 1\right]\,d\mu_n+6\eps.
\end{align*}
As $F_n^k\to F^k$ and $F_n^{k'}\to F^{k'}$ uniformly (cf.\ Step~\ref{st:step4}), we conclude 
\begin{align}\label{eq:FkCauchy1}
\limsup_{n\to\infty }\int  \left[ \left|F^k-F^{k'}\right|\wedge 1\right]\,d\mu_n 
\le 6\eps.
\end{align}
On the other hand, using the mapping theorem and~\eqref{eq:FkregCompare} for both $\mu\equiv\mu_{0}$ and $\mu_{n}$,
\begin{align*}
\int  \left[ \left|F^k-F^{k'}\right|\wedge 1\right]\,d\mu 
&\leq 2\eps + \int  \left[ \left|F_{\mathrm{reg}} ^k-F_{\mathrm{reg}} ^{k'}\right|\wedge 1\right]\,d\mu \\
& = 2\eps + \lim_{n\to\infty} \int  \left[ \left|F_{\mathrm{reg}} ^k-F_{\mathrm{reg}} ^{k'}\right|\wedge 1\right]\,d\mu_{n} \\
&\leq 4\eps + \limsup_{n\to\infty} \int  \left[ \left|F^k-F^{k'}\right|\wedge 1\right]\,d\mu_{n}.
\end{align*}
In view of~\eqref{eq:FkCauchy1}, this yields
$$
\int  \left[ \left|F^k-F^{k'}\right|\wedge 1\right]\,d\mu 
\le 10\eps,
$$
showing that $(F^k)_{k\in \N}$ is Cauchy in $\mu$-probability. In particular, there exists a limit~$f$ in $\mu$-probability. %
\end{step}

\begin{step}\label{st:step7}
  Let $\phi:\R\to[-1,1]$ be uniformly continuous;  we show that $\int \phi(f_{n})\,d\mu_{n}$ converges to $\int \phi(f)\,d\mu$. That is, the law of~$f_{n}$ under~$\mu_{n}$ converges weakly to the law of~$f$ under~$\mu$ \cite[Theorem~2.1]{Billingsley.99}, or in terms of pushforwards, $(f_{n})_{\#}\mu_{n}\to f_{\#}\mu$.
  In particular, if $(f_{n})$ satisfy a normalization $\int \arctan(f_{n})\,d\mu_{n}=\alpha_{n}$ and $\alpha_{n}\to\alpha$, then $\int \arctan(f)\,d\mu=\alpha$.
  
 Let $\eps>0$. In view of~\eqref{eq:claim2} and~\eqref{eq:claim1}, there exists $k_{0}\in\N$ such that for all $k\geq k_{0}$ and all $n\in\N$,
\begin{align*}
    \left|\int \phi(f_{n})\,d\mu_{n} - \int \phi(f)\,d\mu\right|
    & \leq \left|\int \phi(f_{n})\,d\mu_{n} - \int \phi(F_{n}^{k})\,d\mu_{n}\right|\\
    & \quad+  \left|\int \phi(F_{n}^{k})\,d\mu_{n} - \int \phi(F^{k})\,d\mu\right| \\
    & \quad+ \left|\int \phi(F^{k})\,d\mu - \int \phi(f)\,d\mu\right| \\
    & \leq 2\eps + \left|\int \phi(F_{n}^{k})\,d\mu_{n} - \int \phi(F^{k})\,d\mu\right|.
\end{align*} 
Fix $k\geq k_{0}$ such that $\delta_{k}\leq \eps$. Then by~\eqref{eq:FkregCompare},
\begin{align*}
    & \left|\int \phi(F_{n}^{k})\,d\mu_{n} - \int \phi(F^{k})\,d\mu\right| \\
    & \quad \leq \left|\int \phi(F_{n}^{k})\,d\mu_{n} - \int \phi(F^{k})\,d\mu_{n}\right|
      + \left|\int \phi(F^{k})\,d\mu_{n} - \int \phi(F^{k})\,d\mu_{n}\right| \\
    & \quad \leq 2\eps+ \left|\int \phi(F_{n}^{k})\,d\mu_{n} - \int \phi(F^{k})\,d\mu_{n}\right|
      + \left|\int \phi(F^{k}_{\mathrm{reg}})\,d\mu_{n} - \int \phi(F^{k}_{\mathrm{reg}})\,d\mu_{n}\right|.
\end{align*} 
As $F_{n}^{k}\to F^{k}$ uniformly and $F^{k}_{\mathrm{reg}}$ is continuous $\mu$-a.s., there exists $n_{0}=n_{0}(k)$ such that the last line is $\leq 4\eps$ for all $n\geq n_{0}$. In summary, 
$|\int \phi(f_{n})\,d\mu_{n} - \int \phi(f)\,d\mu|\leq 6\eps$ for $n\geq n_{0}$, proving the claim.
\end{step}

Above, we have introduced the functions $f_{n}^{k}, F_{n}^{k},  F^{k},f$ on~$\X$. Analogously, one constructs $g_{n}^{k}, G_{n}^{k}, G^{k},g$ on~$\Y$. We can now detail the main step of the proof, showing that $f,g$ are indeed potentials for a coupling $\pi\in\Pi(\mu,\nu)$. To keep track of the argument more easily, we state the technical parts as lemmas and prove them at the end.

\begin{step}\label{st:step8}
Recall that
$$
 \pi_{n}(dx,dy)=e^{f_{n}(x)+g_{n}(y)-c(x,y)} \,\mu(dx)\,\nu(dy) \in \Pi(\mu_{n},\nu_{n}).
$$
Consider the nonnegative measure 
$$
\pi(dx,dy):=e^{f(x)+g(y)-c(x,y)} \,\mu(dx)\,\nu(dy)
$$ and let $S\subset\X\times\Y$ be measurable with $\pi(\partial S)=0$; we show  $\pi(S)= \lim_{n\to \infty} \pi_n(S)$. %
Define the auxiliary measures 
\begin{align*}
\pi_n^C(dx,dy)&= e^{f_n(x)\wedge C +g_n(y)\wedge C-c(x,y)}\,\mu_n(dx)\,\nu_n(dy),\\
\pi^C(dx,dy)&= e^{f(x)\wedge C+g(y)\wedge C-c(x,y)}\,\mu(dx)\,\nu(dy),\\
\pi^{k,C}(dx,dy)&=e^{F^k(x)\wedge C+G^k(y)\wedge C-c(x,y)}\,\mu(dx)\,\nu(dy).
\end{align*}
Fix $\eps\in (0,1/9)$ and consider the decomposition
\begin{align}\label{eq:long}
| \pi_n(S)-\pi(S)|
&\le  \big|\pi_n(S)-\pi_n^C(S)\big|+ \big| \pi_n^C(S)- \pi^{k,C}(S) \big|+\big|  \pi^{k,C}(S)-\pi^C(S)\big|\nonumber \\
& \quad+ \big| \pi^C(S) -\pi(S)\big|.
\end{align}
We estimate separately the four terms on the right-hand side.

\begin{lemma}\label{lem:bounded2}
We have
\begin{align*}%
\lim_{C\to \infty} \sup_{n\in \N}\big\|\pi_n - \pi_n^C\big\|_{TV}=0.
\end{align*}
\end{lemma}

The lemma is proved at the end of Step~\ref{st:step8}. For the first term in~\eqref{eq:long}, Lemma~\ref{lem:bounded2} shows that there exists $C>0$ such that 
\begin{align}\label{eq:long1}
\sup_{n\in \N}\big|\pi_n(S)-\pi_n^C(S)\big|\le \eps.
\end{align}
We continue with the last term of \eqref{eq:long}. Since $x\mapsto e^x$ is increasing and nonnegative, an application of the monotone convergence theorem for $C\to\infty$ shows that after increasing~$C$ as necessary, we have 
\begin{align}
\big| \pi^C(S) -\pi(S)\big| \le \eps \quad&\mbox{if}\quad \pi(S)<\infty, \label{eq:contr2} \\
\pi^C(S) \ge 2  \quad&\mbox{if}\quad \pi(S)=\infty \label{eq:contr1}.
\end{align}
(The second case will be eliminated by contradiction later on.) The value of~$C$ is now fixed for the remainder of the proof.

Turning to the third term  in~\eqref{eq:long}, note that since $F^k\stackrel{\mu}{\to}f$ and $G^k\stackrel{\nu}{\to} g$,
\begin{align*}
e^{F^k(x)\wedge C+G^k(y)\wedge C-c(x,y)}\to e^{f(x)\wedge C+g(y)\wedge C-c(x,y)} \qquad\mbox{in $\mu\otimes \nu$-probability}.
\end{align*}
As these expressions are uniformly bounded, dominated convergence implies that there exists $k_0\in \N$ with
\begin{align}\label{eq:long3}
\big| \pi^{k,C}(S)-\pi^C(S)\big|\le \eps \qforallq k\ge k_0.
\end{align}

It remains to estimate the second term, $|\pi_n^C(S)-\pi^{k,C}(S)|$, for large~$n$, which is the main difficulty. Choose $k_1\in \N$ such that for all $k\ge k_1$ we have
\begin{align}\label{eq:heat4}
|e^{a}-e^{\tilde{a}}|\le \eps \quad\mbox{whenever $|a-\tilde{a}|\le 4\tilde{\delta}_k$ and $a,\tilde{a}\le 2C$.}
\end{align}
Moreover, choose $k_2\in \N$ such that for all $k\ge k_2$ we have
\begin{align}\label{eq:heat8}
2\delta_{k}e^{2C} \le \eps.
\end{align}
For the remainder of the proof, $k\ge \max(k_0,k_1,k_2)$ is fixed.

\begin{lemma}\label{lem:better1}
For all $k\ge \max(k_0,k_1,k_2)$,
\begin{align*}%
\sup_{n\in \N}\left| \pi_n^C(S)-\int_S e^{F_n^k(x)\wedge C+G_n^k(y)\wedge C-c(x,y)} \mu_n(dx)\,\nu_n(dy) \,\right|\le 2\eps.
\end{align*}
\end{lemma}

\begin{lemma}\label{lem:better3}
For all $k\ge \max(k_0,k_1,k_2)$,
\begin{align*}%
&\limsup_{n\to \infty} \Big|\pi^{k,C}(S)- \int_S e^{F_n^k(x)\wedge C+G_n^k(y)\wedge C-c(x,y)}\,\mu_n(dx)\,\nu_n(dy) \Big|\le 4\eps.
\end{align*}
\end{lemma}
The lemmas are proved at the  end of Step~\ref{st:step8}. Together, they show
\begin{align*}%
&\limsup_{n\to \infty} \big|\pi^C_n(S)-\pi^{k,C}(S)\big|\le 6 \eps.
\end{align*}
Combining this with \eqref{eq:long1} and \eqref{eq:long3} yields
\begin{align}\label{eq:longPre}
\limsup_{n\to \infty}\big| \pi_n(S)-\pi^C(S)\big|\le  8\eps.
\end{align}
In particular, there exists $n\in \N$ such that 
$
\big| \pi_{n}(S)-\pi^C(S)\big|\le  9\eps,
$
and as~$\pi_{n}$ is a probability measure,  it follows that 
$
\pi^C(S)\le 1+9\eps.
$
This contradicts~\eqref{eq:contr1} and hence establishes that we are in the case~\eqref{eq:contr2}. In view of~\eqref{eq:longPre}, that yields
\begin{align*}
\limsup_{n\to \infty}\left| \pi_n(S)-\pi(S)\right|\le  9\eps.
\end{align*}
This shows $\lim_{n\to \infty}\pi_n(S)=\pi(S)$. Thus we have proved that $\pi\in\cP(\X\times\Y)$ and $\pi_n\to \pi$ weakly \cite[Theorem~2.1]{Billingsley.99}. As the marginals then also converge weakly and $\pi_{n}\in\Pi(\mu_{n},\nu_{n})$, this implies $\pi\in\Pi(\mu,\nu)$.
\end{step}

\begin{proof}[Proof of Lemma~\ref{lem:bounded2}.]
Let $\kappa>0$. By~\eqref{eq:bddL1Cond} and Markov's inequality there exists $C>0$ such that $\mu_n\left( f_n\ge C\right) \le \kappa$ and $\nu_n \left( g_n\ge C\right) \le \kappa$ for all $n\in \N$.
For any measurable set $S\subset\X\times\Y$, the definition of $\pi_n^C$ then yields
\begin{align*}
\pi_n(S) \ge  \pi_n^C(S)
&\ge  \pi_n^C(S\cap \{f_n< C\}\cap \{g_n< C\} )\\
&= \pi_n(S\cap \{f_n< C\}\cap \{g_n< C\} )\\
&\ge \pi_n(S)- \mu_n(f_n\ge  C)- \nu_n(g_n\ge C)
\ge \pi_n(S)-2\kappa. 
\end{align*}
Hence $\sup_{n\in\N,\,S\subset\X\times\Y }|\pi_n(S)-\pi_n^C(S)|\leq2\kappa$ and the claim follows.
\end{proof}

\begin{proof}[Proof of Lemma~\ref{lem:better1}.]
We split the integral into
\begin{align*}
&\left| \pi_n^C(S)-\int_S e^{F_n^k(x)\wedge C+G_n^k(y)\wedge C-c(x,y)} \mu_n(dx)\,\nu_n(dy) \,\right|\\
&\le   \int_{S\cap (A_{n}^{k}\times B_{n}^{k})} \left| e^{f_n(x)\wedge C+g_n(y)\wedge C-c(x,y)} -e^{F_n^k(x)\wedge C+G_n^k(y)\wedge C-c(x,y)}\right|\,\mu_n(dx)\,\nu_n(dy)\\
&+  \int_{(A_{n}^{k}\times B_{n}^{k})^c} \left| e^{f_n(x)\wedge C+g_n(y)\wedge C-c(x,y)} - e^{F_n^k(x)\wedge C+G_n^k(y)\wedge C-c(x,y)}\right| \,\mu_n(dx)\,\nu_n(dy).
\end{align*}
To estimate the first integral, recall from~\eqref{eq:new4} that
$
|F_n^k -f_n|\le 2\tilde{\delta}_k
$
on $A_{n}^{k}$. Using also the analogue on $B_{n}^{k}$, \eqref{eq:heat4} implies that the integrand is bounded by~$\eps$ on $A_{n}^{k}\times B_{n}^{k}$. 
Regarding the second integral, the integrand is bounded by $e^{2C}$ and
$$
(\mu_n\otimes\nu_n)\big[(A_{n}^{k}\times B_{n}^{k})^c\big] \leq \mu_{n}((A_{n}^{k})^{c}) + \nu_{n}((B_{n}^{k})^{c})\leq 2\delta_{k}.
$$
In summary,
\begin{align*}
\left| \pi_n^C(S)-\int_S e^{F_n^k(x)\wedge C+G_n^k(y)\wedge C-c(x,y)}  \mu_n(dx)\,\nu_n(dy) \,\right|
\le \eps+ 2\delta_{k}e^{2C} \stackrel{\eqref{eq:heat8}}{\le} 2\eps
\end{align*}
as claimed.
\end{proof}

\begin{proof}[Proof of Lemma~\ref{lem:better3}.]
For brevity, denote $\tilde{\pi}_n(dx,dy):= e^{-c(x,y)}\mu_n(dx)\nu_n(dy)$ for all $n\in \N_{0}$. As $\mu_n\to\mu_{0}$ and $\nu_n\to\nu_{0}$ and $(x,y)\mapsto e^{-c(x,y)}$ is bounded and continuous, we also have $\tilde{\pi}_n \to \tilde{\pi}_{0}$ weakly.
Recall from Step~\ref{st:step4} (and its analogue on~$\Y$ with self-explanatory notation) that
\begin{align*}
F_n^k(x)&=\sum_{j\in \cJ^{k}} f_n^k(x_{j,n}^{k}) \1_{D_{j}^k\cap \Xcpt^{k}}(x),  &G_n^k(x)&=\sum_{l\in \cL^{k}} g_n^k(y_{l,n}^{k}) \1_{E_{l}^k\cap \Ycpt^{k}}(y), \\
F^k(x)&=\sum_{j\in \cJ^{k}} a_j^k \1_{D_{j}^k\cap \Xcpt^{k}}(x), &G^k(y)&=\sum_{l\in\cL^{k}} b_l^k \1_{E_{l}^k\cap \Ycpt^{k} }(y),
\end{align*}
where $\cJ^{k}$ and $\cL^{k}$  are finite sets.
Thus 
\begin{align*}
&\int_{S\cap(\Xcpt^{k}\times \Ycpt^{k} )}   e^{F_n^k(x)\wedge C+G_n^k(y)\wedge C-c(x,y)}\,\mu_n(dx)\,\nu_n(dy)\\
&=\sum_{j\in\cJ^{k}}\sum_{l\in\cL^{k}} e^{f_n^k(x_{j,n}^{k})\wedge C +g_n^k(y_{l,n}^{k})\wedge C} \tilde{\pi}_n\left[S\cap (D_{j}^k\times E_{l}^k)\cap(\Xcpt^{k}\times\Ycpt^{k}) \right]
\end{align*}
and similarly for $F^k,G^k$ instead of $F_n^k,G_n^k$.
By the construction in Step~\ref{st:step4}, 
\begin{align}\label{eq:la1}
\lim_{n\to \infty}e^{f_n^k(x_{j,n}^{k})\wedge C +g_n^k(y_{l,n}^{k})\wedge C}= e^{a_j^k\wedge C+b_l^{k}\wedge C}.
\end{align}
As~$S$ and $(D_{j}^k\times E_{l}^k)$ are $\mu\otimes\nu$-continuity sets, we also have
\begin{align}\label{eq:la2}
\lim_{n\to \infty}\tilde{\pi}_n\left[(D_{j}^k\times E_{l}^k)\cap S\right]= \tilde{\pi}_{0}\left[(D_{j}^k\times E_{l}^k)\cap S\right].
\end{align}
Finally, we note that for all $n\in \N_0$,
\begin{align}\label{eq:add}
\begin{split}
e^{2C} \tilde{\pi}_n\big[(\Xcpt^{k}\times \Ycpt^{k} )^c\big]
&\le e^{2C} (\mu_n\otimes\nu_n)\big[(\Xcpt^{k}\times \Ycpt^{k} )^c\big] 
\le 2e^{2C}\delta_k\stackrel{\eqref{eq:heat8}}{\le} \eps.
\end{split}
\end{align}
We can now expand the difference to be estimated as
\begin{align*}
\begin{split}
& \Big|\pi^{k,C}(S)- \int_S e^{F_n^k(x)\wedge C+G_n^k(y)\wedge C-c(x,y)}\,\mu_n(dx)\,\nu_n(dy) \Big|\\
&\le \Big|\pi^{k,C}(S\cap(\Xcpt^{k}\times \Ycpt^{k} ))- \int_{S\cap(\Xcpt^{k}\times \Ycpt^{k} )} e^{F_n^k(x)\wedge C+G_n^k(y)\wedge C-c(x,y)}\,\mu_n(dx)\,\nu_n(dy) \Big|\\
&\qquad+e^{2C} \Big(\tilde{\pi}_0\big[(\Xcpt^{k}\times \Ycpt^{k} )^c\big]+ \tilde{\pi}_n\big[(\Xcpt^{k}\times \Ycpt^{k} )^c\big]\Big)\\
&\stackrel{\eqref{eq:add}}{\le} 2\eps + \sum_{j\in\cJ^{k}}\sum_{l\in\cL^{k}} \Big|  e^{a_j^k\wedge C+b_l^{k}\wedge C}\tilde{\pi}_0\left[S\cap (D_{j}^k\times E_{l}^k)\cap(\Xcpt^{k}\times\Ycpt^{k}) \right]\\
&\hspace{10em}- e^{f_n^k(x_{j,n}^{k})\wedge C +g_n^k(y_{l,n}^{k})\wedge C}\tilde{\pi}_n\left[S\cap (D_{j}^k\times E_{l}^k)\cap(\Xcpt^{k}\times\Ycpt^{k})\right]\Big|\\
&\stackrel{\eqref{eq:add}}{\le} 4\eps + \sum_{j\in\cJ^{k}}\sum_{l\in\cL^{k}} \Big|  e^{a_j^k\wedge C+b_l^{k}\wedge C}\tilde{\pi}_0\left[S\cap (D_{j}^k\times E_{l}^k) \right]\\
&\hspace{10em} - e^{f_n^k(x_{j,n}^{k})\wedge C +g_n^k(y_{l,n}^{k})\wedge C}\tilde{\pi}_n\left[S\cap (D_{j}^k\times E_{l}^k) \right]\Big|.
\end{split}
\end{align*}
In view of~\eqref{eq:la1} and~\eqref{eq:la2}, taking $n\to\infty$ yields
$$
\limsup_{n\to \infty} \Big|\pi^{k,C}(S)- \int_S e^{F_n^k(x)\wedge C+G_n^k(y)\wedge C-c(x,y)}\,\mu_n(dx)\,\nu_n(dy) \Big|\le 4\eps.  \qedhere
$$
\end{proof}

\begin{step}\label{st:step9}
In Step~\ref{st:step8} we have shown that 
$$
  \pi(dx,dy):=e^{f(x)+g(y)-c(x,y)} \,\mu(dx)\,\nu(dy)
$$
defines a coupling of $\mu,\nu$. Moreover, Step~\ref{st:step7} and~\eqref{eq:posBddL1Cond} imply that $f^{+}\in L^{1}(\mu)$ and $g^{+}\in L^{1}(\nu)$. 
By the general verification result in Proposition~\ref{pr:duality}, the form of $\pi$ with $(f\oplus g)^{+}\in L^{1}(\mu\otimes\nu)$ implies that~$(f,g)\in L^{1}(\mu)\times L^{1}(\nu)$, that $\cC(\mu,\nu)<\infty$ and that $\pi=\pi_{*}$ is the unique minimizer for the entropic optimal transport problem~\eqref{eq:EOT}. 
It follows that $\pi_{n}\to\pi_{*}$ also holds along the original sequence, completing the proof of Theorem~\ref{thm:main}\,(i).

The potentials corresponding to~$\pi_{*}$ are a.s.\ uniquely determined up to an additive constant  (Proposition~\ref{pr:existence}) and we have seen in Step~\ref{st:step7} that $\alpha_{n}=\int \arctan(f_{n})\,d\mu_{n}\to\int \arctan(f)\,d\mu$. As 
$\int \arctan(f+a)\,d\mu$ is strictly increasing in $a\in\R$, it follows that $(f,g)$ are the unique potentials with normalization $\int \arctan(f)\,d\mu=\alpha$. In particular, $(f,g)$ do not depend on the subsequence chosen in Step~\ref{st:step4} and the proof of Theorem~\ref{thm:main}\,(ii) is complete.
\end{step}

\begin{step}\label{st:step10} It remains to prove Theorem~\ref{thm:main}\,(iii).
Let~\eqref{eq:uniform_int} hold. Passing to a subsequence, we may assume that $\alpha_{n}\to\alpha$ and $f,g$ are as in~Theorem~\ref{thm:main}\,(ii).
 We first show the upper semicontinuity
\begin{align}\label{eq:UIlimitUSC}
\limsup_{n\to \infty}\int f_n\,d\mu_n\le \int f\,d\mu, \qquad \limsup_{n\to \infty} \int g_n\,d\nu_n\le \int g\,d\mu.
\end{align}
Indeed, the weak convergence~\eqref{eq:claimConvInDistrib} and the uniform integrability~\eqref{eq:uniform_int} imply that $\lim_{n\to \infty}\int f^{+}_n\,d\mu_n=\int f^{+}\,d\mu$. Together with Portmanteau's theorem  for $(f^{-}_n)$, the first part of~\eqref{eq:UIlimitUSC} follows, and similarly for the second.

Next, we argue the lower semicontinuity of the sum,
\begin{align}\label{eq:lowercont}
\liminf_{n\to \infty} \left(\int f_n\,d\mu_n +\int g_n\,d\nu_n \right)&\ge \int f\,d\mu +\int g\,d\nu.
\end{align}
By~\eqref{eq:posBddL1Cond} and Proposition~\ref{pr:duality}, we have the duality
\begin{align*}
\inf_{\pi\in\Pi(\mu_n,\nu_n)}\int c(x,y)\,\pi(dx,dy)+H(\pi| \mu_n\otimes \nu_n)= \int f_n\,d\mu_n +\int g_n\,d\nu_n .
\end{align*}
Similarly, $(f,g)\in L^{1}(\mu)\times L^{1}(\nu)$ implies the duality for the limiting problem. As a consequence, it suffices to argue lower semicontinuity in the primal problem. Let $\pi_{n}=\argmin_{\pi\in\Pi(\mu_n,\nu_n)}\int c(x,y)\,\pi(dx,dy)+H(\pi| \mu_n\otimes \nu_n)$. Then $(\pi_{n})_{n\in\N}$ is tight and any weak cluster point belongs to~$\Pi(\mu,\nu)$. Using the lower semicontinuity of $(\pi',\mu',\nu')\mapsto \int c \,d\pi' +H(\pi'| \mu'\otimes \nu')$, cf.\ \cite[Lemma~1.3]{Nutz.20}, we deduce
\begin{align*}
\liminf_{n\to \infty} \left( \int f_n\,d\mu_n +\int g_n\,d\nu_n \right)
&=\liminf_{n\to \infty} \left( \inf_{\pi\in\Pi(\mu_n,\nu_n)}\int c\,d\pi+H(\pi| \mu_n\otimes \nu_n)\right) \\
& =\liminf_{n\to \infty} \int c\,d\pi_{n}+H(\pi_{n}| \mu_{n}\otimes \nu_{n})\\
& \ge \inf_{\pi\in \Pi(\mu,\nu)} \int c\,d\pi+H(\pi| \mu\otimes \nu) \\
&= \int f\,d\mu +\int g\,d\nu.
\end{align*}

Together, the lower semicontinuity~\eqref{eq:lowercont} of the sum and the separate upper semicontinuity~\eqref{eq:UIlimitUSC} imply~\eqref{eq:meansConverge}.
This completes the proof of Theorem~\ref{thm:main}. \qed
\end{step}

\appendix
\section{Background on Entropic Optimal Transport}\label{se:background}

Let $(\mu,\nu)\in\cP(\X)\times\cP(\Y)$ and let $c:\X\times \Y\to [0,\infty)$ be measurable. We have the following result on existence and uniqueness for the entropic optimal transport problem~\eqref{eq:EOT}. 

\begin{proposition}\label{pr:existence}
  If $\cC(\mu,\nu)<\infty$, there is a unique minimizer $\pi_{*}\in\Pi(\mu,\nu)$ for~\eqref{eq:EOT}. Moreover, $\pi_{*}\sim \mu\otimes\nu$ and there are measurable functions $f: \X\to\R$, $ g: \Y\to\R$, called potentials, such that
  $$
    \frac{d\pi_{*}}{d(\mu\otimes\nu)} = e^{f\oplus g -c }\quad \mu\otimes\nu\as
  $$
  The potentials are a.s.\ unique up to an additive constant: if $f', g'$ are potentials, then $f'=f+a$ $\mu$-a.s.\ and $g'= g-a$ $\nu$-a.s.\ for some~$a\in\R$. 
  
  If $c\in L^{1}(\mu\otimes\nu)$, then $\cC(\mu,\nu)<\infty$ and $(f, g)\in L^{1}(\mu)\times L^{1}(\nu)$.
\end{proposition}

See  \cite[Theorem~4.2]{Nutz.20} for a proof. Conversely, the next result shows that the form of the density characterizes the minimizer. We also include the duality relation.
  
\begin{proposition}\label{pr:duality}
Let $\pi_{0}\in\Pi(\mu,\nu)$ admit a density of the form
$$
  \frac{d\pi_{0}}{d(\mu\otimes\nu)} = e^{f_{0}\oplus g_{0} -c}\quad \mu\otimes\nu\as
$$
for some measurable functions $f_{0}: \X\to\R$ and $g_{0}: \Y\to\R$. 
\begin{itemize}
\item[(a)] If $\cC(\mu,\nu)<\infty$, then $\pi_{0}$ is the minimizer~$\pi_{*}$ and $f_{0}, g_{0}$ are its potentials. 

\item[(b)] If $(f_{0}\oplus g_{0})^{+}\in L^{1}(\mu\otimes\nu)$, then necessarily $(f_{0}, g_{0})\in L^{1}(\mu)\times L^{1}(\nu)$ and
\begin{equation}\label{eq:duality}
  \cC(\mu,\nu) = \int f_{0}\,d\mu + \int g_{0}\,d\nu.
\end{equation}
  In particular, $\cC(\mu,\nu)<\infty$ and~(a) applies.
\end{itemize} 
\end{proposition}

See \cite[Theorem~4.2, Theorem~4.7, Remark~4.8]{Nutz.20}. 
When $f,g$ are potentials as in Proposition~\ref{pr:existence}, the fact that $\pi_{*}\in\Pi(\mu,\nu)$ implies the so-called Schr\"odinger equations
\begin{align}\label{eq:SE}
\begin{split}
    f(x)&= -\log \int_{\Y} e^{g(y)-c(x,y)}\, \nu(dy) \quad \mu\mbox{-a.s.},\\
    g(y) &= -\log \int_{\X} e^{f(x)-c(x,y)} \, \mu(dx) \quad \nu\mbox{-a.s.}
\end{split}
\end{align} 
We may choose versions of the potentials such that these relations hold without exceptional sets.

\section{Uniform Stability under Strong Integrability}\label{se:unifStability}

The following result shows that stability of the potentials, even uniformly on compacts, can be obtained quite easily when the cost is sufficiently integrable. As discussed in the Introduction, the integrability condition is \emph{not} satisfied in the regime of principal interest. For the statement, we choose versions of the potentials such that~\eqref{eq:SE} holds without exceptional sets.

\begin{proposition}\label{pr:unifStability}
For $n\in \N$, let $(\mu_n, \nu_n)\in \mathcal{P}(\mathcal{X})\times \mathcal{P}(\mathcal{Y})$ satisfy $\cC(\mu_{n},\nu_{n})<\infty$ and let $(f_{n},g_{n})$ be corresponding potentials. Suppose that $\mu_n,\nu_n$ converge weakly to~$\mu,\nu$ and
\begin{align}\label{eq:unifStabilityCond}
\sup_{n\in\N} \int e^{\beta f_n}\,d\mu_n<\infty, \qquad \sup_{n\in\N} \int e^{\beta g_n}\,d\nu_n<\infty \qquad \mbox{
for some $\beta> 1$.}
\end{align}
Then $f_{n}\oplus g_{n}\to f\oplus g$ uniformly on compacts, where $(f,g)$ are arbitrary potentials for~$(\mu,\nu)$, and the corresponding optimal couplings converge weakly.
If $\alpha_{n},\alpha,f,g$ are as in~Theorem~\ref{thm:main}\,(ii), then also $f_n\to f$ and $g_n\to g$, uniformly on compacts.
\end{proposition}

Similarly as in Lemma~\ref{le:suff_cond_for_bddL1}, a sufficient condition for~\eqref{eq:unifStabilityCond} is that
\begin{align}\label{eq:condition_cost_exp_int}
\sup_{n\in\N} \int e^{\beta c}\,d (\mu_n\otimes \nu_n)<\infty
\end{align}
and $(f_{n},g_{n})$ are normalized such that~\eqref{eq:negBddL1Cond} holds, for instance $\int f_{n}\,d\mu_{n}=0$.

\begin{proof}[Proof of Proposition~\ref{pr:unifStability}.]
We first show that $(f_{n})_{n\in\N}$ is pointwise bounded. 
Let $C=\sup_n (\int e^{\beta g_n}\,d\nu_n)^{1/\beta}$. As $c\geq0$, we have by~\eqref{eq:SE} that
\begin{align*}
   e^{-f_{n}(x)} &=\int e^{g_{n}(y)-c(x,y)}\,\nu_{n}(dy)
\le \int e^{g_{n}(y)}\,\nu_{n}(dy) \leq C.
\end{align*}
This yields the uniform lower bound $f_{n}\geq -\log C$, and similarly $g_{n}\geq -\log C$. 
Fix $x\in\X$. Using~\eqref{eq:SE} and the lower bound, 
\begin{align*}
e^{-f_n(x)}=  \int e^{g_n(y)-c(x,y)}\,\nu_n(dy)\ge C^{-1} \int e^{-c(x,y)}\,\nu_n(dy).
\end{align*}
As $e^{-c(x,\cdot)}$ is bounded and continuous and $\nu_n \to \nu$ weakly, $\int e^{-c(x,y)}\,\nu_n(dy)$ converges to $\int e^{-c(x,y)}\,\nu(dy)>0$. We conclude that $\sup_{n\in\N} f_{n}(x)<\infty$, completing the proof that $(f_{n})$ is pointwise bounded.

Next, we show that $(f_{n})$ is equicontinuous. On the strength of the pointwise boundedness, it suffices to show that $e^{-f_n}$ is equicontinuous. Fix a compatible metric on $\X$, let $x_{0}\in\X$ and $\eps>0$. Using tightness, choose a compact $K\subset\Y$ with $\nu_{n}(K^{c})<\eps$ for all $n$. Let $q\in(1,\infty)$ satisfy $1/\beta +1/q=1$. By H\"older's inequality,
\begin{align*}
|e^{-f_n(x_0)}-e^{-f_n(x)}| &=\left| \int e^{g_n(y)-c(x_0,y)}\,\nu_n(dy)-\int e^{g_n(y)-c(x,y)}\,\nu_n(dy)\right|\\
&\le \left(\int e^{\beta g_n(y)}\,\nu_n(dy)\right)^{1/\beta } \left(\int \left|e^{-c(x_0,y)}-e^{-c(x,y)}\right|^q\,\nu_n(dy)\right)^{1/q}\\
&\le C\left(\int \left|e^{-c(x_0,y)}-e^{-c(x,y)}\right|^q\,\nu_n(dy)\right)^{1/q}.
\end{align*}
Recalling that $e^{-c}\le 1$, the integral can be estimated by
\begin{align*}
\int \left|e^{-c(x_0,y)}-e^{-c(x,y)}\right|^q\,\nu_n(dy)\le \int_K \left|e^{-c(x_0,y)}-e^{-c(x,y)}\right|^q\,\nu_n(dy)+2\nu_n(K^c).
\end{align*}
As~$c$ is continuous and $K$ is compact, $\sup_{y\in K}\left|e^{-c(x_0,y)}-e^{-c(\cdot,y)}\right|$ is continuous. Therefore, $\int_K |e^{-c(x_0,y)}-e^{-c(x,y)} |^q\,\nu_n(dy)<\eps$ for $x\in B_{\delta}(x_{0})$, for $\delta>0$ sufficiently small, and we obtain the desired equicontinuity,
\begin{align*}
|e^{-f_n(x_0)}-e^{-f_n(x)}| \leq C ( \eps + 2\eps )^{1/q}, \qquad  x\in B_{\delta}(x_{0}).
\end{align*}

We have shown that  $(f_{n})$ is equicontinuous and pointwise bounded, and the same arguments hold for $(g_{n})$. Passing to a subsequence, the Arzel\`a--Ascoli theorem shows that $f_n\to f$ and $g_n\to g$ uniformly on compacts. 
Note that $(e^{f_{n}\oplus g_{n} -c})_{n}$ is $(\mu_{n}\otimes\nu_{n})_{n}$-uniformly integrable by~\eqref{eq:unifStabilityCond}. As $e^{f_{n}\oplus g_{n} -c}\to e^{f\oplus g -c}$ uniformly on compacts and $\mu_{n}\otimes\nu_{n}\to\mu\otimes\nu$ weakly, it follows for the optimal couplings~$\pi_{n}$ that
$$
  \pi_{n} = e^{f_{n}\oplus g_{n} -c}\,d(\mu_{n}\otimes\nu_{n})\,\to\, e^{f\oplus g -c}\,d(\mu\otimes\nu)=:\pi_{0} \qquad\mbox{weakly.}
$$
In particular, $\pi_{0}\in\Pi(\mu,\nu)$. Proposition~\ref{pr:duality} now shows that $\pi_{0}$ is the optimal coupling for $(\mu,\nu)$ and $(f,g)$ are corresponding potentials. As $f\oplus g$ is unique (Proposition~\ref{pr:existence}), the claim for the original sequence follows.
\end{proof}

\begin{remark}\label{rk:unifStability}
  Proposition~\ref{pr:unifStability} extends to the boundary case $\beta=1$ if $e^{-c}$ is uniformly continuous. This holds in particular when~$c(x,y)=\|x-y\|^{2}$ on~$\R^{d}\times\R^{d}$, or more generally whenever~$c$ is coercive on~$\X\times\Y$.
  
  Following the above proof, the argument for pointwise boundedness still applies, and the argument for equicontinuity is even simpler under the additional hypothesis. The argument using uniform integrability may no longer be clear, but we can instead use Theorem~\ref{thm:main} to conclude that~$(f,g)$ must be potentials for~$(\mu,\nu)$.
\end{remark} 
\section{Omitted Proofs}\label{se:omittedProofs}

\begin{proof}[Proof of Lemma~\ref{le:suff_cond_for_bddL1}.]
(i) Let $\inf_{n\in\N} \int g_n\,d\nu_n\geq-C$ for $C\geq0$. By~\eqref{eq:SE} and Jensen's inequality,
\begin{align*}
f_n(x) &=-\log\left( \int e^{g_n(y)-c(x,y)}\,\nu_n(dy)\right) \\
&\le \int \left[ c(x,y)-g_n(y) \right]\,\nu_n(dy)\le C+ \int c(x,y)\,\nu_n(dy).
\end{align*}
Recalling $c\geq0$, this shows 
\begin{equation}\label{eq:proof_suff_cond_for_bddL1}
f_n^+(x)\le C+ \int c(x,y)\,\nu_n(dy),
\end{equation}
and now Tonelli's theorem yields
$
\int f_n^+ \,\mu_n \le C+ \int c \,d(\mu_n\otimes\nu_{n}).
$
The analogue holds for $g_{n}$. Thus~\eqref{eq:condition_cost_int} implies~\eqref{eq:posBddL1Cond}, and via \eqref{eq:negBddL1Cond} also~\eqref{eq:bddL1Cond}.

More generally, given a measurable set $A\subset\X$, \eqref{eq:proof_suff_cond_for_bddL1} also implies 
$$
\int_{A} f_n^+ \,\mu_n \le C\mu_{n}(A)+ \int_{A\times \Y} c \,d(\mu_n\otimes\nu_{n}).
$$
If $\mu_{n}(A)\leq\delta$, then $(\mu_n\otimes \nu_n)(A\times\Y)\leq \delta$, so  that the $\eps$--$\delta$ characterization of uniform integrability  yields the claim.

(ii) The variational representation of relative entropy \cite[Lemma~1.3]{Nutz.20} shows that 
  $$
     H(\mu_{n}\otimes\nu_{n}|\mu\otimes\nu) \geq \int \psi  \,d(\mu_{n}\otimes\nu_{n}) - \log  \int e^{\psi}\,d(\mu\otimes\nu)
  $$
  for any measurable function $\psi$ bounded from below.
  Choosing $\psi=\beta c$, we deduce 
  $$
     \int \beta c  \,d(\mu_{n}\otimes\nu_{n}) \leq 
     H(\mu_{n}\otimes\nu_{n}|\mu\otimes\nu)   + \log  \int e^{\beta c}\,d(\mu\otimes\nu).
  $$
  Noting that $H(\mu_{n}\otimes\nu_{n}|\mu\otimes\nu)=H(\mu_{n}|\mu)+H(\nu_{n}|\nu)$, the right-hand side is bounded under~\eqref{eq:condition_entropy}, so that~\eqref{eq:condition_cost_int} applies. To obtain the last claim, we replace~$c$ with $\phi(c)$ in the preceding argument and apply the la Vall\'ee--Poussin theorem.
\end{proof}

\begin{proof}[Proof of Lemma~\ref{le:TVconvImpliesCond}]
  Let $C>0$ and $A_{n}=\{\frac{d\mu}{d\mu_n}\ge C\}$. Then
  $
  \mu(A_{n})=\int_{A_{n}} \frac{d\mu}{d\mu_{n}}\,d\mu_{n}\ge C\mu_{n}(A_{n})
  $
  and thus $\mu_{n}(A_{n})\leq C^{-1}$ for all $n\in\N$. Writing
  $$
    \mu(A_{n}) = \mu(A_{n}) - \mu_{n}(A_{n})+\mu_{n}(A_{n}) \leq \|\mu - \mu_{n}\|_{TV} + C^{-1}
  $$
  then shows that $\limsup_{n}\mu(A_{n})\leq C^{-1}$. The claim follows.
\end{proof}

\begin{proof}[Proof of Corollary~\ref{co:TVstability}.]
  By Corollary~\ref{cor:abs_cont} and Lemma~\ref{le:TVconvImpliesCond}, we have $f_{n}\to f$ in $\mu$-probability and $g_{n}\to g$ in $\nu$-probability after passing to a subsequence. Consider the Lebesgue decomposition $\mu_{n}=\mu_{n}'+\mu_{n}''$ into $\mu_{n}'\ll\mu$ and $\mu_{n}''\bot\mu$. Then $\mu_{n}'\to\mu$ in total variation and hence $\frac{d\mu_{n}'}{d\mu}\to1$ in $L^{1}(\mu)$. This implies the convergence in $\mu$-probability of the reciprocal,
 and as $\frac{d\mu}{d\mu_{n}'}=\frac{d\mu}{d\mu_{n}}$ $\mu$-a.s., that
$\frac{d\mu}{d\mu_{n}}\to1$ in $\mu$-probability.
Similarly, $\frac{d\nu}{d\nu_{n}}\to1$ in $\nu$-probability.
Following the proof of the particular case in Section~\ref{se:mainResults} but writing the reciprocals,
 $$
   \frac{d(\mu\otimes\nu)}{d\pi_{n}} = \frac{d(\mu_{n}\otimes\nu_{n})}{d\pi_{n}}\frac{d\mu}{d\mu_{n}}\frac{d\nu}{d\nu_{n}} = e^{-(f_{n}\oplus  g_{n}-c)}\frac{d\mu}{d\mu_n}\frac{d\nu}{d\nu_n} 
   \,\to\, e^{-(f\oplus  g-c)} = \frac{d(\mu\otimes\nu)}{d\pi_{*}}
 $$
 in $\mu\otimes\nu$-probability. Consider the Lebesgue decomposition $\pi_{n}=\pi_{n}'+\pi_{n}''$ into $\pi_{n}'\ll\mu\otimes\nu$ and $\pi_{n}''\bot\mu\otimes\nu$. Then it follows that
 $$
   \frac{d\pi_{n}'}{d(\mu\otimes\nu)} = \left(\frac{d(\mu\otimes\nu)}{d\pi_{n}}\right)^{-1}\1_{\big\{\frac{d(\mu\otimes\nu)}{d\pi_{n}}\neq0\big\}} \to \frac{d\pi_{*}}{d(\mu\otimes\nu)}
 $$
 in $\mu\otimes\nu$-probability, which by Scheff\'e's lemma implies $\pi_{n}'\to\pi_{*}$ in total variation. As $\pi_{n},\pi_{*}$ are probability  measures, it follows that $\pi_{n}''(\X\times\Y)\to0$  and finally $\pi_{n}\to\pi_{*}$ in total variation.  The convergence of the original sequence follows.
\end{proof}

\begin{proof}[Proof of Lemma~\ref{lem:ext}.]
Fix $\eps\in (0,\delta)$, to be determined later. By Prokhorov's theorem, we can find compacts $\Xcpt$ and $\Ycpt$ with
$\inf_{n\in \N_0}\mu_n(\Xcpt)\ge 1-\eps^2/2$ and $\inf_{n\in \N_0}\nu_n(\Ycpt)\ge 1-\eps^2/2$. As  $\pi_{n}\in \Pi(\mu_n,\nu_n)$, it follows that
\begin{align}\label{eq:tightness}
\inf_{n\in \N}\pi_{n}(\Xcpt \times \Ycpt)\ge 1-\eps^2.
\end{align}
Fix $n\in \N$ and consider the set $$A_{n}=\left\{x\in \Xcpt:\ \int_{\Ycpt} e^{f_n(x)+g_n(y)-c(x,y)}\,\nu_n(dy) \ge 1-\eps\right\};$$
we claim that its complement satisfies
\begin{align}\label{eq:markov}
p_n:=\mu_n\left(A_{n}^c\right)\le \eps.
\end{align}
Indeed, \eqref{eq:SE} yields
\begin{align}\label{eq:exponential}
\int_{\Ycpt} e^{f_n(x)+g_n(y)-c(x,y)}\,\nu_n(dy)\le \int e^{f_n(x)+g_n(y)-c(x,y)}\,\nu_n(dy)=1
\end{align}
and thus
\begin{align*}
1-\eps^2 &\stackrel{\eqref{eq:tightness}}{\le}\pi_n(\Xcpt \times \Ycpt)= \int_{\Xcpt}\int_{\Ycpt} e^{f_n(x)+g_n(y)-c(x,y)} \,\nu_n(dy)\mu_n(dx)\\
&\le \int_{A_n^c}\int_{\Ycpt} e^{f_n(x)+g_n(y)-c(x,y)} \,\nu_n(dy)\mu_n(dx)\\
&\quad +\int_{A_n}\int_{\Ycpt} e^{f_n(x)+g_n(y)-c(x,y)} \,\nu_n(dy)\mu_n(dx)\\
&\stackrel{\eqref{eq:exponential}}{\le} (1-\eps)p_n+(1-p_n)=1-p_n\eps,
\end{align*}
which implies~\eqref{eq:markov}. Next, we observe from the definition of $A_n$ and \eqref{eq:exponential} that for $x\in A_n$,
\begin{align}\label{eq:ineq1}
-\left( \log \int_{\Ycpt} e^{g_n(y)-c(x,y)}\,\nu_n(dy)-\log(1-\eps)\right)&\le f_n(x) \nonumber\\
&\le 
-\log \int_{\Ycpt} e^{g_n(y)-c(x,y)}\,\nu_n(dy).
\end{align}
Let $x_1, x_2\in A_n$ and assume without loss of generality that $f_n(x_1)\ge f_n(x_2)$. Then
\begin{align}\label{eq:longest}
\begin{split}
\left|f_n(x_1)-f_n(x_2)\right|&\stackrel{\eqref{eq:ineq1}}{\le}  \left(\log \int_{\Ycpt} e^{g_n(y)-c(x_2,y)}\, \nu_n(dy)-\log(1-\eps)\right)\\
&\quad - \log \int_{\Ycpt} e^{g_n(y)-c(x_1,y)}\, \nu_n(dy) \\
&= \log \int_{\Ycpt} e^{c(x_1,y)-c(x_2,y)+g_n(y)-c(x_1,y)}\, \nu_n(dy) -\log(1-\eps)\\
&\qquad- \log \int_{\Ycpt} e^{g_n(y)-c(x_1,y)}\, \nu_n(dy)\\
&\le  \log \left( e^{\sup_{y\in \Ycpt} \left|c(x_1, y)-c(x_2,y)\right| } \int_{\Ycpt} e^{g_n(y)-c(x_1,y)}\, \nu_n(dy) \right) \\
&\qquad-\log(1-\eps)- \log \int_{\Ycpt} e^{g_n(y)-c(x_1,y)}\, \nu_n(dy)\\
&=\sup_{y\in \Ycpt} \left|c(x_1, y)-c(x_2,y)\right| -\log(1-\eps).
\end{split}
\end{align}
This concludes the proof of the first estimate in the lemma. 
Turning to the second, note that by \eqref{eq:tightness}, \eqref{eq:markov} and the definition of $A_n$,
\begin{align}\label{eq:nextround}
\begin{split}
\pi_n(A_n \times \Ycpt )
&\ge  \pi_n(\Xcpt \times \Ycpt )-\pi_n( \X\setminus A_n \times \Ycpt )\\
&\ge  1-\eps^2-\int_{A_n^c} \int_{\Y_{\text{cpt}}} e^{f_n(x)+g_n(y)-c(x,y)} \,\nu_n(dy)\,\mu_n(dx)\\
&\ge  1-\eps^2-\eps(1-\eps)
=1-\eps=1-\delta^2,
\end{split}
\end{align}
where we chose $\eps:=\delta^2$ (ensuring  $\eps \in (0,\delta)$, in particular). Define \begin{align*}
B_{n}=\left\{y\in \Ycpt:\ \int_{A_n} e^{f_n(x)+g_n(y)-c(x,y)}\,\mu_n(dx) \ge 1-\delta\right\}.
\end{align*}
Arguing as for~\eqref{eq:markov} and~\eqref{eq:ineq1}, now using \eqref{eq:nextround} instead of~\eqref{eq:tightness}, we see that $\nu_n(B_{n}^{c}) \le \delta$ and that for $y\in B_{n}$,
\begin{align*}%
\begin{split}
-\left( \log \int_{A_n} e^{f_n(x)-c(x,y)}\,\mu_n(dx)-\log(1-\delta)\right)&\le g_n(y)\\
&\le 
-\log \int_{A_n} e^{f_n(x)-c(x,y)}\,\mu_n(dx).
\end{split}
\end{align*}
We conclude the proof by arguing as in~\eqref{eq:longest} but with $f_n,\eps$ replaced by $g_n,\delta$.
\end{proof}

\newcommand{\dummy}[1]{}

\end{document}